\documentclass[12pt]{amsart}
 \usepackage[dvips]{graphicx}
\usepackage{amsmath, amssymb}
\usepackage{pifont}

  \numberwithin{equation}{section}
  \newtheorem{theorem}{Theorem}[section]  
  \newtheorem{theorem?}{``Theorem''}[section]  
  \newtheorem{corollary}[theorem]{Corollary}
  \newtheorem{proposition}[theorem]{Proposition}
  \newtheorem{lemma}[theorem]{Lemma}

\theoremstyle{definition}

\theoremstyle{remark}
  \newtheorem{remark}[theorem]{Remark}  
  \newtheorem{note}[theorem]{Note}

\newcommand{\R}{{\mathbb R}}
\newcommand{\C}{{\mathbb C}}
\newcommand{\Q}{{\mathbb Q}}
\newcommand{\N}{{\mathbb N}}
\newcommand{\Z}{{\mathbb Z}}
\newcommand{\1}{{\bf 1}}
\renewcommand{\a}{\alpha}
\renewcommand{\b}{\beta}
\renewcommand{\d}{\partial}
\renewcommand{\1}{{\bf 1}}
\newcommand{\card}{{}^{\Pisymbol{psy}{"23}}\!}

\begin{document}
\title[Oscillatory integrals]
{
Asymptotic analysis of 
oscillatory integrals 
via the Newton polyhedra of 
the phase and the amplitude
} 
\author{Koji Cho, Joe Kamimoto and Toshihiro Nose}
\address{Faculty of Mathematics, Kyushu University, 
Motooka 744, Nishi-ku, Fukuoka, 819-0395, Japan} 
\email{
cho@math.kyushu-u.ac.jp}
\email{
joe@math.kyushu-u.ac.jp}
\email{ 
t-nose@math.kyushu-u.ac.jp}
\keywords{oscillatory integrals, 
oscillation index and its multiplicity, 
local zeta function,  
asymptotic expansion, 
Newton polyhedra of the phase and the amplitude, 
essential set.}
\subjclass[2000]{58K55 (14B05, 14M25).}
\thanks{
The second author was supported by 
Grant-in-Aid for Scientific Research (C) (No. 22540199), 
Japan Society for the Promotion of Science. 
}
\thanks{This paper will be published in J. Math. Soc. Japan.}
\maketitle



\begin{abstract}
The asymptotic behavior at infinity 
of oscillatory integrals is in detail investigated
by using the Newton polyhedra of the phase and 
the amplitude.
We are especially interested in 
the case that the amplitude has a zero 
at a critical point of the phase. 
The properties of poles of local zeta functions, 
which are closely related
to the behavior of oscillatory integrals, 
are also studied under the associated situation. 
\end{abstract}




\section{Introduction}
In this paper, we investigate the asymptotic behavior of 
oscillatory integrals, that is, integrals of the form
\begin{equation}
I(\tau)=
\int_{\R^n}e^{i\tau f(x)}\varphi(x)\chi(x)dx,
\label{eqn:1.1}
\end{equation}
for large values of the real parameter $\tau$, 
where $f$, $\varphi$, $\chi$ are real-valued smooth functions 
defined on $\R^n$ and     
$\chi$ is a cut-off function with small support 
which identically equals one 
in a neighborhood of the origin in $\R^n$.  
Here $f$ and $\varphi\chi$ are 
called the {\it phase} and the {\it amplitude}, 
respectively. 

By the principle of stationary phase, 
the main contribution in the
behavior of the integral (\ref{eqn:1.1}) 
as $\tau\to+\infty$ is given by the local properties 
of the phase around its critical points.
We assume that the phase has 
a critical point at 
the origin,  i.e., $\nabla f(0)=0$. 
The following deep result has been obtained 
by 
using Hironaka's resolution of singularities 
\cite{hir64}
(cf. \cite{mal74}). 
If $f$ is real analytic on a neighborhood of the origin
and
the support of $\chi$ is contained 
in a sufficiently small neighborhood of 
the origin, then
the integral $I(\tau)$ has an
asymptotic expansion of the form
\begin{equation}
I(\tau)\sim
e^{i\tau f(0)}
\sum_{\alpha} \sum_{k=1}^{n} 
C_{\alpha k}\tau^{\alpha} (\log\tau)^{k-1} \quad 
\mbox{as $\tau\to +\infty$}, 
\label{eqn:1.2}
\end{equation}
where 
$\alpha$ runs through a finite number of 
arithmetic progressions, not depending on 
$\varphi$ and $\chi$, 
which consist of negative rational numbers. 
Our interest focuses the largest $\alpha$ 
occurring in (\ref{eqn:1.2}). 
Let $S(f,\varphi)$ be the set 
of pairs $(\alpha,k)$ such that 
for each neighborhood of the origin in $\R^n$, 
there exists a cut-off function $\chi$ 
with support in this neighborhood 
for which $C_{\alpha k}\neq 0$ in the asymptotic 
expansion (\ref{eqn:1.2}).  
We denote by $(\beta(f,\varphi), \eta(f,\varphi))$
the maximum of the set $S(f,\varphi)$ 
under the lexicographic ordering, 
i.e. $\beta(f,\varphi)$ is the maximum of
values $\alpha$ for which we can find 
$k$ so that $(\alpha,k)$ belongs to $S(f,\varphi)$;
$\eta(f,\varphi)$ is the maximum of integers
$k$ satisfying that $(\beta(f,\varphi),k)$ belongs to 
$S(f,\varphi)$. 
We call $\beta(f,\varphi)$   
{\it oscillation index}  
of $(f,\varphi)$ and 
$\eta(f,\varphi)$ its {\it multiplicity}. 
(This multiplicity, less one, is equal to 
the corresponding multiplicity 
in \cite{agv88},p.183.)

From various points of view, 
the following is an interesting problem: 
What kind of information of the phase and 
the amplitude determines (or estimates) 
the oscillation index $\beta(f,\varphi)$ and 
its multiplicity $\eta(f,\varphi)$?  
There have been many interesting studies 
concerning this problems
(\cite{var76},\cite{ds89},\cite{sch91},\cite{ds92},\cite{dns05},\cite{gre09},\cite{gre10},\cite{grea07}, etc.).  
In particular, 
the significant work of Varchenko \cite{var76} 
shows the following by 
using the theory of toric varieties: 
By the geometry of the Newton polyhedron of $f$, 
the oscillation index can be estimated and, 
moreover, this index and its multiplicity can be 
exactly determined when $\varphi(0)\neq 0$, 
under a certain nondegenerate condition of the phase 
(see Theorem~2.1 in Section~2).  
Since his study, 
the investigation of the behavior of 
oscillatory integrals has been more closely linked with 
the theory of singularities. 
Refer to the excellent expositions 
\cite{agv88},\cite{kan81} for studies in this direction.  
Besides \cite{var76}, recent works of Greenblatt 
\cite{gre04},\cite{gre08},\cite{gre09},\cite{gre10},\cite{grea07} 
are also interesting. 
He explores a certain resolution of singularities, 
which is obtained from an elementary method, 
and investigates the asymptotic behavior of $I(\tau)$. 
His analysis is also available 
for a wide class of phases 
without the above nondegenerate condition.

In this paper, 
we generalize and improve the above results of 
Varchenko \cite{var76}. 
To be more precise, 
we are especially interested in the behavior of 
the integral (\ref{eqn:1.1}) as $\tau\to +\infty$  
when $\varphi$ has a zero at a critical 
point of the phase.  
Indeed, under some assumptions, 
we obtain more accurate results 
by using the Newton polyhedra of not only the phase 
but also the amplitude. 
Closely related issues have been investigated by 
Arnold, Gusein-Zade and Varchenko \cite{agv88} and 
Pramanik and Yang \cite{py04}, and 
they obtained similar results to ours. 
From the point of view of our investigations, 
their results will be reviewed 
in Remark~2.8 in Section~2 and Section~7.4. 
In our results, 
delicate geometrical conditions of the Newton polyhedra
of the phase and the amplitude affect the behavior of 
oscillatory integrals.  
There exist some faces of the Newton polyhedron
of the amplitude, 
which play a crucial role in determining 
the oscillation index and its multiplicity. 
Furthermore, 
in order to determine the oscillation index in general, 
we need not only geometrical properties 
of their Newton polyhedra but also
information about the coefficients 
of the terms, corresponding to the above faces,  
in the Taylor series of the amplitude. 
(See Theorem~2.7 in Section~2.2 and Example~2 in 
Section~7.3.) 

It is known  
(see, for instance, 
\cite{igu78},\cite{agv88},\cite{kan81}, 
and Section~6.1 in this paper) 
that 
the asymptotic analysis of oscillatory integral (\ref{eqn:1.1}) 
can be reduced to an investigation of the poles of the functions 
$Z_+(s)$ and $Z_-(s)$ (see (\ref{eqn:5.1}) below), 
which are similar to the local zeta function 
\begin{equation}
Z(s)=
\int_{\R^n} |f(x)|^s\varphi(x)\chi(x)dx,
\label{eqn:1.3}
\end{equation}
where $f$, $\varphi$, $\chi$ are the same 
as in (\ref{eqn:1.1}) with $f(0)=0$. 
The substantial analysis in this paper is to  
investigate the properties of poles of 
the local zeta function $Z(s)$ and 
the functions $Z_{\pm}(s)$ by using 
the Newton polyhedra of the functions 
$f$ and $\varphi$.
See Section~5 for more details. 

Many problems in analysis, 
including  
partial differential equations, 
mathematical physics, harmonic analysis and 
probability theory, 
lead to the need to study the behavior of 
oscillatory integrals of the form (\ref{eqn:1.1})
as $\tau\to+\infty$. 
We explain the original motivation for our investigation. 
In the function theory of several complex variables, 
it is an important problem to understand
boundary behavior of 
the Bergman kernel 
for pseudoconvex domains. 
In \cite{joe04}, the special case of domains of finite type 
is considered and 
the behavior as $\tau\to+\infty$ of the Laplace integral
$$
\tilde{I}(\tau)=
\int_{\R^n}e^{-\tau f(x)}\varphi(x)dx
$$
plays an important role 
in boundary behavior of the above kernel. 
Here $f$,$\varphi$ are $C^{\infty}$ functions 
satisfying certain conditions.   
The computation of asymptotic expansion 
of the above kernel in \cite{joe04} requires 
precise analysis of $\tilde{I}(\tau)$ 
when $\varphi$ 
has a zero at the critical point of $f$. 
Our analysis in this paper can be applied 
to the case of the above Laplace integrals.  
See also \cite{cko04},\cite{ckn11}.

This paper is organized as follows. 
In Section~2, after explaining  
some important notions and terminology,
we state main results relating to oscillatory integrals. 
In Section~3, 
we consider an important assumption 
in Theorem~2.7 in Section~2, 
which is related to elementary convex geometry
(cf. \cite{zie95}). 
In Section~4, 
we overview the theory of toric varieties and 
explain a certain resolution of singularities.  
In Section~5,
we investigate the properties of 
poles of the local zeta function 
$Z(s)$ and 
the functions $Z_{\pm}(s)$ by using 
the resolution of singularities constructed 
in Section~4.
In Section~6, 
we give proofs of theorems on the behavior
of oscillatory integrals stated 
in Section~2. 
In Section~7, 
we give some examples, which clarify the subtlety 
of our results.
Lastly, we check a related result in \cite{agv88}
with these examples.

{\it Notation and Symbols.}\quad
\begin{itemize}
\item 
We denote by $\Z_+, \Q_+, \R_+$ the subsets consisting of 
all nonnegative numbers in $\Z,\Q,\R$, respectively.
\item
We use the multi-index as follows.
For $x=(x_1,\ldots,x_n),y=(y_1,\ldots,y_n) \in\R^n$, 
$\a=(\a_1,\ldots,\a_n)\in\R_+^n$, 
define
\begin{eqnarray*}
&& 
|x|=\sqrt{|x_1|^{2}+\cdots +|x_n|^{2}}, \quad 
 \langle x,y \rangle =x_1 y_1+\dots+x_n y_n, 
\\
&& 
x^{\a}=x_1^{\a_1}\cdots x_n^{\a_n}, \quad
 \langle \a \rangle =\a_1+\cdots+\a_n.  
\end{eqnarray*}
\item
For $A,B\subset \R^n$ and $c\in\R$, 
we set 
$$
A+B=\{a+b\in\R^n; a\in A \mbox{ and } b\in B\},\quad
c\cdot A=\{ca\in\R^n; a\in A\}.
$$
\item
We express by $\1$ the vector $(1,\ldots,1)$ or the set 
$\{(1,\ldots,1)\}$. 
\item
For a finite set $A$, 
$\card A$ means the cardinality of $A$. 
\item
For a $C^{\infty}$ function $f$, 
we denote by Supp($f$) the support of $f$, i.e., 
Supp($f$)$=\overline{\{x\in \R^n; f(x)\neq 0\}}$. 
\end{itemize}


\section{Definitions and main results}

\subsection{Newton polyhedra}
Let us explain some necessary notions 
to state our main theorems. 
The definitions of more fundamental
terminologies (polyhedra, faces, dimensions, etc.)
will be given in Section~3.1.

Let $f$ 
be a real-valued $C^{\infty}$ function defined 
on a neighborhood of the origin in $\R^n$, 
which has the Taylor series   
$
\sum_{\alpha\in{\Z}_+^n} c_{\alpha}x^{\alpha} 
$ 
at the origin. 
Then, the {\it Taylor support} of $f$ is the set   
$
S_f=\{\a\in\Z_+^n;c_{\a}\neq 0\} 
$ 
and the {\it Newton polyhedron} of $f$
is the integral polyhedron: 
$$
\Gamma_+(f)=
\mbox{the convex hull of the set 
$\bigcup \{\a+\R_+^n;\a\in S_f\}$ in $\R_+^n$}
$$
(i.e., the intersection 
of all convex sets 
which contain $\bigcup \{\a+\R_+^n;\a\in S_f\}$).
The union of the compact faces of 
the Newton polyhedron $\Gamma_+(f)$ is called 
the {\it Newton diagram} $\Gamma(f)$ of $f$, 
while the topological boundary of $\Gamma_+(f)$ 
is denoted by 
$\d\Gamma_+(f)$. 
The {\it principal part} of $f$ is defined by 
$f_0(x)=\sum_{\alpha\in\Gamma(f)\cap\Z_+^n}
c_{\alpha}x^{\alpha}.$
For a compact subset 
$\gamma\subset\partial\Gamma_+(f)$, 
let 
$f_{\gamma}(x)=\sum_{\a\in\gamma\cap\Z_+^n}
c_{\alpha}x^{\alpha}$.
$f$ is said to be {\it nondegenerate} 
over $\R$ with respect to the Newton polyhedron 
$\Gamma_+(f)$ if
for every compact face $\gamma\subset\Gamma(f)$, 
the polynomial $f_{\gamma}$ satisfies 
$$
\nabla f_{\gamma}=
\left(
\frac{\d f_{\gamma}}{\d x_1},\ldots,
\frac{\d f_{\gamma}}{\d x_n}\right)
\neq (0,\ldots,0)\quad
\mbox{on the set $\{x\in\R^n; x_1\cdots x_n\neq 0\}$.}
$$
$f$ is said to be {\it convenient} if 
the Newton diagram $\Gamma(f)$ 
intersects all the coordinate axes. 

Let $f,\varphi$ be real-valued $C^{\infty}$ functions 
defined on a neighborhood of the origin in $\R^n$ 
and assume that $\Gamma(f)$ and 
$\Gamma(\varphi)$ are nonempty. 
We define the {\it Newton distance} of $(f,\varphi)$ by 
\begin{equation}
d(f,\varphi)=
\min\{d>0;
d\cdot(\Gamma_+(\varphi)+\1)\subset\Gamma_+(f)\}.
\label{eqn:2.1}
\end{equation}
It is easy to see 
$d(f,\varphi)=\max\{d>0;
\d\Gamma_+(f)\cap d\cdot(\Gamma_+(\varphi)+\1)\neq\emptyset\}.$
The number $d(f,\varphi)$ corresponds to what is called 
the {\it coefficient of inscription} of $\Gamma_+(\varphi)$
in $\Gamma_+(f)$ in \cite{agv88}, p 254. 
(Their definition in \cite{agv88} must be slightly modified.)
Let $\Gamma(\varphi,f)$ be the subset in $\R^n$
defined by 
$$
\Gamma(\varphi,f)+\1=
\left(
\frac{1}{d(f,\varphi)}\cdot\d\Gamma_+(f)
\right)
\cap
(\Gamma_+(\varphi)+\1).
$$
In the above definition, 
$\partial\Gamma_+(\varphi)$ can be used 
instead of $\Gamma_+(\varphi)$  
(see Remark~3.2). 
Lemma~3.1, below, implies that 
$\Gamma(\varphi,f)$ is some union of faces of 
$\Gamma_+(\varphi)$.  

Let $\Gamma^{(k)}$ be the union of $k$-dimensional faces of 
$\Gamma_+(f)$. 
Then $\Gamma_+(f)$ is stratified as
$
\Gamma^{(0)}\subset\Gamma^{(1)}\subset\cdots\subset
\Gamma^{(n-1)}(=\d\Gamma_+(f))
\subset
\Gamma^{(n)}(=\Gamma_+(f)).
$
Let $\tilde{\Gamma}^{(k)}=
\Gamma^{(k)}\setminus\Gamma^{(k-1)}$ for 
$k=1,\ldots,n$ and 
$\tilde{\Gamma}^{(0)}=\Gamma^{(0)}$. 
A map $\rho_f:\Gamma_+(f)\to\{0,1,\ldots,n\}$ 
is defined as 
$\rho_f(\a)=k$ if $\a\in\tilde{\Gamma}^{(n-k)}$.   
In other words, 
$\rho_f(\alpha)$ is the codimension of 
the face of $\Gamma_+(f)$, 
whose relative interior contains the point $\alpha$.
We define the {\it Newton multiplicity} of $(f,\varphi)$ by  
$$
m(f,\varphi)=\max\{
\rho_f(d(f,\varphi)(\a+\1)); \mbox{
$\a\in \Gamma(\varphi,f)$}
\}.
$$
Let $\Gamma_0$ be the subset of $\Gamma(\varphi,f)$ defined by
$$
\Gamma_0=\{\a\in\Gamma(\varphi,f);
\rho_f(d(f,\varphi)(\a+\1))=m(f,\varphi)\},
$$
which is called the {\it essential set} on 
$\Gamma(\varphi,f)$.
Proposition~3.3, below, shows that
$\Gamma_0$ is a disjoint union of faces of 
$\Gamma_+(\varphi)$.

Consider the case 
$\varphi(0)\neq 0$. 
Then $\Gamma_+(\varphi)=\R_+^n$. 
In this case, $d(f,\varphi)$ and $m(f,\varphi)$ 
are denoted by $d_f$ and $m_f$, respectively. 
(Note that $d(f,\varphi)\leq d_f$ for general $\varphi$.)  
It is easy to see that the point $q=(d_f,\ldots,d_f)$ is the 
intersection of the line $\alpha_1=\cdots=\alpha_n$ in $\R^n$
and $\d\Gamma_+(f)$, and that $m_f=\rho_f(q)$.
$\Gamma(\varphi,f)=\Gamma_0=\{0\}.$
More generally, in the case that 
$\Gamma_+(\varphi)=\{p\}+\R_+^n$ with $p\in\Z_+^n$, 
the geometrical meanings of 
the quantities $d(f,\varphi)$ and $m(f,\varphi)$ 
will be considered in Proposition 5.4 below.

\subsection{Main results}

Let us explain our results relating to the behavior 
of the oscillatory integral $I(\tau)$ 
in (\ref{eqn:1.1}) as $\tau\to+\infty$. 

Throughout this subsection, 
$f$, $\varphi$, $\chi$ satisfy the following conditions:
Let $U$ be an open neighborhood of the origin in $\R^n$.
\begin{enumerate}
\item[(A)] 
$f:U \to \R$ is a real analytic function satisfying 
that $f(0)=0$, $|\nabla f(0)|=0$ and $\Gamma(f)\neq\emptyset$;
\item[(B)] 
$\varphi:U \to \R$ is a $C^{\infty}$ function
satisfying 
$\Gamma(\varphi)\neq\emptyset$;
\item[(C)]
$\chi:\R^n\to\R_+$ is a $C^{\infty}$ function 
which identically
equals one in some neighborhood of the origin 
and has a  small support which 
is contained in $U$. 
\end{enumerate}

As mentioned in the Introduction, 
it is known that 
the oscillatory integral (\ref{eqn:1.1}) 
has an asymptotic expansion of the form (\ref{eqn:1.2}).
Before stating our results, 
we recall a part of famous results due to 
Varchenko in \cite{var76}. 
In our language,  they are stated as follows. 
\begin{theorem}[Varchenko \cite{var76}]
Suppose that $f$ is nondegenerate
over $\R$ with respect to its Newton polyhedron. 
Then
\begin{enumerate}
\item[(i)] $\beta(f,\varphi)\leq -1/d_f$ for any $\varphi$;
\item[(ii)] If $\varphi(0)\neq 0$ and $d_f>1$, then 
$\beta(f,\varphi)=-1/d_f$ and $\eta(f,\varphi)=m_f$;
\item[(iii)]
The progression $\{\alpha\}$ in $(\ref{eqn:1.2})$
belongs to finitely many 
arithmetic progressions, 
which are obtained by using the theory
of toric varieties based on the geometry of 
the Newton polyhedron $\Gamma_+(f)$. 
$($See Remark~$2.6$, below.$)$
\end{enumerate}
\end{theorem}

Now, let us explain our results. 
First, we investigate more precise situation in 
the estimate in the part (i) of Theorem~2.1. 
Indeed, when $\varphi$ has a zero at the origin, 
the oscillation index $\beta(f,\varphi)$ can be 
more accurately estimated 
by using the Newton distance $d(f,\varphi)$, 
which is called ``the coefficient of inscription of 
$\Gamma_+(\varphi)$ in $\Gamma_+(f)$'' in \cite{agv88}. 


\begin{theorem}
Suppose that 
{\rm (i)} $f$ is nondegenerate over $\R$ 
with respect to its Newton polyhedron 
and 
{\rm (ii)} 
at least one of the following conditions is satisfied:
\begin{enumerate}
\item[(a)] $f$ is convenient;
\item[(b)] $\varphi$ is convenient;
\item[(c)] $\varphi$ is real analytic on $U$;
\item[(d)] $\varphi$ is expressed as 
$\varphi(x)=x^p\tilde{\varphi}(x)$ 
on $U$, where 
$p\in\Z_+^n$ and $\tilde{\varphi}$ is a $C^{\infty}$ function 
defined on $U$ with $\tilde{\varphi}(0)\neq 0$.
\end{enumerate}  
Then, we have
$\beta(f,\varphi)\leq -1/d(f,\varphi)$.
\end{theorem}
\begin{remark}
A more precise estimate for $I(\tau)$ 
is obtained as follows:
If the support of $\chi$ is
contained in a sufficiently small neighborhood 
of the origin, 
then 
there exists a positive constant $C$ independent of 
$\tau$ such that 
\begin{equation*}
|I(\tau)|\leq C\tau^{-1/d(f,\varphi)}
(\log \tau)^{A-1}
\quad \mbox{
for $\tau\geq 1$},
\label{eqn:}
\end{equation*}
where 
\begin{equation*}
A:=\begin{cases}
m(f,\varphi)& \quad \mbox{if $1/d(f,\varphi)$ 
is not an integer}, \\
\min\{m(f,\varphi)+1, n\}&
\quad \mbox{otherwise}.
\end{cases}
\end{equation*} 
The details will be explained 
in the proof of the above theorem in Section~6.
\end{remark}
\begin{remark}
Let us consider the above theorem 
under the assumptions (i), (ii)-(d) without 
the condition: $\tilde{\varphi}(0)\neq 0$. 
Then the estimate  
$\beta(f,\varphi)\leq -1/d(f,\varphi)$ 
does not always hold. 
In fact, consider the two-dimensional example: 
$f(x_1,x_2)=x_1^2$, 
$\varphi(x_1,x_2)=x_1^2(x_1^2+ e^{-1/x_2^2})$.
The proof of Theorem~2.2, however, implies that 
the estimate $\beta(f,\varphi)\leq -1/d(f,x^p)$ holds
under the above assumptions. 
This assertion with $p=(0,\ldots,0)$ shows  
the assertion (i) in Theorem~2.1. 

Vassiliev \cite{vas79} obtained a similar result 
to that in the case of (d).
\end{remark}
\begin{remark}
The condition (d) implies 
$\Gamma_+(\varphi)=\{p\}+\R_+^n$. 
When $\varphi$ is a $C^{\infty}$ function, however, 
the converse is not true in general. 
We give an example in Section~7.2, 
which shows that the assumption (d) cannot be replaced by 
the condition: 
$\Gamma(\varphi)=\{p\}+\R_+^n$ 
in Theorem~2.2. 
\end{remark}
\begin{remark}
From the proof of the above theorem,  
we can see that  
under the same condition, 
the progression $\{\alpha\}$ in (\ref{eqn:1.2}) is 
contained in the set 
$$
\left\{
-\frac{\tilde{l}(a)+ \langle a \rangle +\nu}{l(a)};a\in\tilde{\Sigma}^{(1)},
\nu\in\Z_+
\right\}\cup(-\N),
$$
where the symbols 
$l(a)$, $\tilde{l}(a)$ and  
$\tilde{\Sigma}^{(1)}$ are
as in Theorem~5.10, below.
This explicitly shows the assertion (iii) in Theorem~2.1. 
\end{remark}

Next, let us give an analogous result to the part (ii) in 
Theorem 2.1, due to Varchenko. 
Indeed, the following theorem deals with the case that 
the equation 
$\beta(f,\varphi)=-1/d(f,\varphi)$ holds.

\begin{theorem} 
Suppose that 
{\rm (i)} $f$ is nondegenerate over $\R$ 
with respect to its Newton polyhedron, 
{\rm (ii)} 
at least one of the following two conditions is satisfied: 
\begin{enumerate}
\item[(a)]
$d(f,\varphi)>1$;
\item[(b)]
$f$ is nonnegative or nonpositive on $U$,
\end{enumerate}
and 
{\rm (iii)}
at least one of the following two conditions is satisfied: 
\begin{enumerate}
\item[(c)] 
$\varphi$ is expressed as $\varphi(x)=x^p\tilde{\varphi}(x)$ 
on $U$, 
where every component of $p\in\Z_+^n$ is even and 
$\tilde{\varphi}$ is a $C^{\infty}$ function defined on $U$ 
with $\tilde{\varphi}(0)\neq 0$;
\item[(d)] 
$f$ is convenient and  
$\varphi_{\Gamma_0}$ is nonnegative or nonpositive 
on $U$. 
\end{enumerate}
Then  
the equations 
$\beta(f,\varphi)=-1/d(f,\varphi)$ and 
$\eta(f,\varphi)= m(f,\varphi)$ hold. 
\end{theorem}
\begin{remark} 
Considering the assumptions: 
(i), (ii)-(a), (iii)-(c) with $p=(0,\ldots,0)$ 
in the above theorem, we see 
the assertion (ii) in Theorem~2.1. 

Pramanik and Yang \cite{py04} obtained a similar result 
in the case that 
the dimension is two  
and $\varphi(x)=|g(x)|^{\epsilon}$ where
$g$ is real analytic and $\epsilon$ is positive. 
Their result 
in Theorem 3.1 (a) 
does not need any additional assumptions. 
We explain this reason roughly. 
They use the {\it weighted Newton distance}, 
whose definition is different from our Newton distance. 
The definition of their distance is more intrinsic 
and is based on 
a good choice of coordinate system, 
which induces a clear resolution of singularity. 
Moreover, the nonnegativity of $\varphi$ 
implies the positivity of the coefficient of 
the expected leading term of the asymptotic expansion 
(\ref{eqn:1.2}). 
On the other hand, 
in our case, the corresponding coefficient 
possibly vanishes without the assumption (c) or (d). 
See Sections 7.1 and 7.3. 
\end{remark}
\begin{remark}
In Section~3, 
we discuss the set $\Gamma_0$ and 
the function $\varphi_{\Gamma_0}$ in the condition (d) 
in detail. 
If $\varphi(0)=0$ and 
$\varphi$ takes the local minimal
(resp. the local maximal) at the origin, 
then $\varphi_{\Gamma_0}$ is nonnegative 
(resp. nonpositive) on some neighborhood of the origin. 
\end{remark}
\begin{remark}
It is easy to show that 
Theorem 2.7 can be rewritten in a slightly stronger 
form by replacing the condition (c) by 
the following (c$'$): 

\begin{enumerate}
\item[(c$'$)] $\varphi$ is expressed as 
$\varphi(x)=
\sum_{j=1}^{l} x^{p_j}\tilde{\varphi}_j(x)$
on $U$, where $p_j\in\Z_+^n$ and 
$\tilde{\varphi}_j\in C^{\infty}(U)$ for all $j$ satisfies 
that if $p_j\in\Gamma_0$, then 
every component of $p_j$ is even and 
$\tilde{\varphi}_j(0)>0$ 
(or $\tilde{\varphi}_j(0)<0$) for all $j$. 
\end{enumerate}

We will give an example in Section~7.3, 
which satisfies the conditions (a), (d) 
but does not satisfy the condition (c$'$). 
(Consider the case that the parameter $t$ satisfies 
$0<|t|<2$ in the example.)  
\end{remark}
\begin{remark}
In the one-dimensional case, the conditions (c) and (d)
are equivalent.
\end{remark}

Lastly, let us discuss a ``symmetrical'' property 
with respect to the phase and the amplitude.  
Observe the one-dimensional case. 
Let $f,\varphi$ satisfy that  
$f(0)=f'(0)=\cdots=f^{(q-1)}(0)=
\varphi(0)=\varphi'(0)=\cdots=\varphi^{(p-1)}(0)=0$ and 
$f^{(q)}(0)\varphi^{(p)}(0)\neq 0$, 
where $p,q\in\N$ are even. 
Applying the computation in Chapter 8 in \cite{ste93}
(see also Section 7.1 in this paper), 
we can see that if the support of $\chi$ is sufficiently 
small, then
$$
\int_{-\infty}^{\infty} 
e^{i\tau xf(x)} \varphi(x)\chi(x)dx 
\sim 
\tau^{-\frac{p+1}{q+1}}
\sum_{j=0}^{\infty} C_j \tau^{-j/(q+1)}
\quad\quad \mbox{as $\tau\to\infty$,}
$$
where $C_0$ is a nonzero constant. 
Note that the above expansion can be obtained 
for $C^{\infty}$ functions $f$ and $\varphi$.
In particular, 
$\beta(x f,\varphi)=-\frac{p+1}{q+1}$ holds. 
Similarly, we can get 
$\beta(x \varphi,f)=-\frac{q+1}{p+1}$. 
From this observation, 
the following question seems interesting: 
When does the equality
$\b(x^{\1}f,\varphi)\b(x^{\1}\varphi,f)=1$ hold 
in higher dimensional case?
The following theorem is concerned with this question. 
\begin{theorem}
Let $f$, $\varphi$ be nonnegative or nonpositive real analytic
functions defined on $U$. 
Suppose that both $f$ and $\varphi$ are 
convenient and nondegenerate over $\R$ with respect to 
their Newton polyhedra. 
Then we have
$\b(x^{\1}f,\varphi)\b(x^{\1}\varphi,f)\geq 1$.
Moreover, the following two conditions are equivalent:  
\begin{enumerate}
\item $\b(x^{\1}f,\varphi)\b(x^{\1}\varphi,f)=1$; 
\item There exists a positive rational number $d$ such that 
$\Gamma_+(x^{\1} f)=d\cdot\Gamma_+(x^{\1} \varphi)$. 
\end{enumerate}
If the condition {\rm (i)} or {\rm (ii)} is satisfied, 
then we have
$\eta(x^{\1}f,\varphi)=\eta(x^{\1}\varphi,f)=n$.
\end{theorem}

\section{Convex polyhedra and essential sets}

\subsection{Polyhedra}
Let us give precise definitions for polyhedra, faces, 
dimensions and so on. Refer to \cite{zie95}, etc.   
for general theory of convex polyhedra.  

A ({\it convex}) {\it polyhedron} is  
an intersection of closed halfspaces:
a set $P\subset\R^n$ presented in the form
$$
P=\bigcap_{j=1}^m \{x\in\R^n; \langle a^j,x \rangle \geq z_j \},
$$
for some $a^1,\ldots,a^m\in\R^n$ and 
$z_1,\ldots,z_m \in\R$.
It is known (cf. \cite{zie95}) that the Newton polyhedron 
$\Gamma_+(f)$ in Section 2.1 is a polyhedron. 

Let $P$ be a polyhedron in $\R^n$. 
A pair $(a,z)\in \R^n\times\R$ is {\it valid} for $P$ 
if a linear inequality $ \langle a,x \rangle \geq z$ is 
satisfied for all points $x\in P$. 
For $(a,z)\in \R^n\times\R$, define 
\begin{equation}
H(a,z)=\{x\in\R^n; \langle a,x \rangle =z\}.
\label{eqn:3.1}
\end{equation} 
A {\it face} of $P$ is any set of the form 
$
F=P\cap H(a,z),
$
where $(a,z)$ is valid for $P$. 
Since $(0,0)$ is always valid, 
we consider $P$ itself as a trivial face of $P$;
the other faces are called {\it proper faces}.  
Conversely, 
it is easy to see that any face is a polyhedron. 
Considering the valid pair $(0,-1)$, 
we see that the empty set is always a face of $P$. 
The {\it dimension} of a face $F$ is the dimension of 
its affine hull of $F$ 
(i.e., the intersection of all affine flats that 
contain $F$). 
The faces of dimensions $0,1$ and $\dim(P)-1$
are called {\it vertices}, {\it edges} and 
{\it facets}, respectively. 
The {\it boundary} of a polyhedron $P$, denoted by 
$\d P$,  
is the union of all proper faces of $P$.  
For a face $F$, $\d F$ is similarly defined. 

\subsection{Essential sets}

Let us consider the properties of 
$\Gamma(\varphi,f)$ and the essential 
set $\Gamma_0$ defined in Section~2.1. 
Moreover, we consider the condition (d) in Theorem~2.7.

\begin{lemma}
Let $P_1,P_2$ be $n$-dimensional polyhedra in $\R^n$. 
If $P_1\subset P_2$, then 
$P_1 \cap \partial P_2$ is 
the union of proper faces of $P_1$.
\end{lemma}
\begin{proof}
There exist finite pairs 
$(a^1,z_1),\ldots,(a^l,z_l)\in\R^n\times\R$ 
such that 
every $(a^j,z_j)$ is valid for $P_2$ and 
$
\partial P_2 
=\bigcup_{j=1}^l 
(P_2\cap H(a^j,z_j)).
$
\begin{equation}
\begin{split}
&
P_1 \cap \partial P_2
=
P_1\cap\left[
\bigcup_{j=1}^l 
(P_2 \cap H(a^j,z_j))
\right]
\\
&\quad\quad 
=\bigcup_{j=1}^l 
(
P_1 \cap P_2\cap H(a^j,z_j)
) 
=\bigcup_{j=1}^l 
(P_1\cap H(a^j,z_j)). 
\end{split}
\label{eqn:3.2}
\end{equation}
Since every $(a^j,z_j)$ is also valid for $P_1$ and 
$P_1\cap H(a^j,z_j)$ is a proper
face of $P_1$, then we get the lemma. 
\end{proof}
Hereafter in this section, 
we assume that  
$f,\varphi$ are $C^{\infty}$ functions 
defined on a neighborhood of the origin 
and their Newton polyhedra are nonempty. 

By applying the above lemma to the case:
\begin{equation}
P_1=\Gamma_+(\varphi),\quad 
P_2=\frac{1}{d(f,\varphi)}\cdot\Gamma_+(f)-\1,
\label{eqn:3.3}
\end{equation}
we see that  
$\Gamma(\varphi,f)=P_1\cap \d P_2 
(\neq \emptyset)$ 
is the union of faces 
of $\Gamma_+(\varphi)$.

\begin{remark}
From (\ref{eqn:3.2}),
we see $P_1\cap \partial P_2 \subset \partial P_1$. 
Thus, 
$\d P_1 \cap \d P_2=P_1 \cap \d P_2$ holds. 
\end{remark}

\begin{proposition}
There exist the faces 
$\gamma_1,\ldots,\gamma_l$ of $\Gamma_+(\varphi)$ 
such that 
$$\Gamma_0=\biguplus_{j=1}^l \gamma_j\quad\quad 
(\mbox{disjoint union}).$$ 
Moreover, the dimension of $\gamma_{j}$ is not 
greater than $n-m(f,\varphi)$ for any $j$.
\end{proposition}
\begin{proof}
Set $P_1$ and $P_2$ as in (\ref{eqn:3.3}) and  
let $k_0=n-m(f,\varphi)$.

In the case $k_0=0$, $\Gamma_0$ is the set of 
vertices of $P_1$, which implies the proposition.  
Consider the case $1\leq k_0\leq n$.
Let $F_1,\ldots,F_l$ be the $k_0$-dimensional faces of 
$P_2$ such that $F_j\cap P_1\neq\emptyset$.
Now, let us show the sets $\gamma_j:=F_j\cap P_1$
satisfy the condition in the proposition.
Of course, the union of all $\gamma_j$ is $\Gamma_0$
and the dimensions of $\gamma_j$ are not greater than 
$k_0 (=n-m(f,\varphi))$ for any $j$. 
It suffices to show that 
each $\gamma_j$ is a face of $P_1$ and 
that $\gamma_j\cap\gamma_k=\emptyset$ if $j\neq k$. 
For each $j$, 
there is a pair $(a^j,z_j)\in\R^{n}\times\R$ such that 
it is valid for $P_2$ and $F_j=P_2\cap H(a^j,z_j)$. Thus, 
$\gamma_j=P_1\cap F_j=P_1\cap P_2 \cap H(a^j,z_j)=
P_1\cap H(a^j,z_j)$, 
which implies $\gamma_j$ is a face of $P_1$.
Next, 
by the minimality of $k_0$, 
$\gamma_j$ is contained in the relative interior
of $F_j$ (i.e., $\gamma_j\subset F_j\setminus \d F_j$). 
Since all relative interiors of $F_j$
are disjoint, we have 
$\gamma_j\cap\gamma_k=\emptyset$ 
if $j\neq k$. 
\end{proof}

\begin{lemma}
Let $\gamma$ be a compact face of $\Gamma_+(\varphi)$. 
If $\varphi$ is nonnegative $($or nonpositive$)$ in a 
neighborhood of the origin, 
so is $\varphi_{\gamma}$.
\end{lemma}
\begin{proof}
A pair 
$(a,z)=((a_1,\ldots,a_n),z)\in\R^n\times\R$ corresponds to 
$\gamma$, i.e., 
$H(a,z)\cap\Gamma_+(\varphi)=\gamma$. 
Taylor's formula implies that for any $N\in\N$, 
$\varphi$ can be expressed as  
\begin{equation}
\varphi(x)=\sum_{\a\in S_{\varphi}\cap U_N} c_{\a} x^{\a}
+\sum_{p\in\Z_+^n,  \langle p \rangle =N} x^p\varphi_p(x), 
\label{eqn:3.4}
\end{equation}
where $U_N:=\{\a\in\R_+^n; \langle \a \rangle <N\}$, 
$c_{\a}$ are constants and 
$\varphi_p$ are $C^{\infty}$ functions defined on a 
neighborhood of the origin. 
Here, take a sufficiently large $N$ such that 
$\gamma$ is contained in the set $U_N$.
For $\xi=(\xi_1,\ldots,\xi_n) \in (-\epsilon, \epsilon)^n$, 
$t\in(-\epsilon, \epsilon)$, where 
$\epsilon>0$ is small, a simple computation gives 
\begin{eqnarray*}
&&\varphi(\xi_1 t^{a_1},\ldots,\xi_n t^{a_n})\\
&&
=\sum_{\a\in S_{\varphi}\cap U_N} c_{\a}\xi^{\a}t^{ \langle a,\a \rangle }
+\sum_{p\in\Z_+^n,  \langle p \rangle =N} 
\xi^{\a}t^{ \langle a,p \rangle }\varphi_p(\xi_1 t^{a_1},\ldots,\xi_n t^{a_n})\\
&&
=t^{z}(\varphi_{\gamma}(\xi)+a(\xi,t)t),
\end{eqnarray*}
where $a(\xi,t)$ is a $C^{\infty}$ function 
defined on a neighborhood of $(0,0)\in \R^n\times\R$. 
From the above, it is easy to show the lemma. 
\end{proof}
\begin{remark}
Even if $\varphi_{\gamma}$ is nonnegative 
(resp. nonpositive) near the origin
for every faces $\gamma$ of $\Gamma_+(\varphi)$, 
$\varphi$ is not always nonnegative (resp. nonpositive)
near the origin:
Consider the example  
$\varphi(x_1,x_2)=(x_1-x_2)^2-x_2^4$.
\end{remark}

\begin{proposition}
Let $\gamma_1,\ldots,\gamma_l$ be the faces of 
$\Gamma_+(\varphi)$ as in 
Proposition~$3.3$ and  
suppose $\Gamma_0$ is compact. 
Then the following two conditions are equivalent: 
\begin{enumerate}
\item $\varphi_{\Gamma_0}$ is nonnegative 
$($resp. nonpositive$)$ 
near the origin;
\item $\varphi_{\gamma_j}$ is nonnegative 
$($resp. nonpositive$)$
near the origin for all $j$.
\end{enumerate}
\end{proposition}
\begin{proof}
From Lemma~3.4, we can see that (i) implies (ii).  
Since $\Gamma_0$ is the disjoint union of the faces 
$\gamma_j$, we have 
$\varphi_{\Gamma_0}(x)=\sum_{j=1}^l \varphi_{\gamma_j}(x)$. 
This shows that (ii) implies (i).
\end{proof}
\begin{corollary}
If $f$ is convenient and 
$\varphi$ is nonnegative or nonpositive near the origin, 
then the condition {\rm (d)} in Theorem~$2.7$ is satisfied.  
\end{corollary}
\begin{proof}
The convenience of $f$ implies the compactness of $\Gamma_0$.
By Lemma~3.4, 
the assertion (ii) in Proposition~3.6 is satisfied.  
\end{proof}

\section{Toric resolution}

The purpose of this section is to give the resolution 
of the singularities of the critical points 
of some functions from the theory of toric varieties. 
Refer to 
\cite{kkms73},\cite{oda88},\cite{ful93},\cite{oka97}, etc.  
for general theory of toric varieties. 


\subsection{Cones and fans}
In order to construct a toric resolution 
obtained 
from the Newton polyhedron, 
we recall the definitions of important terminology:  
{\it cone} and {\it fan}.  

A {\it rational polyhedral cone} 
$\sigma\subset \R^n$ is a cone
generated by finitely many elements of $\Z^n$. 
In other words, 
there are $u_1,\ldots,u_k \in \Z^n$ such that  
$$ 
\sigma=\{\lambda_1 u_1+\cdots+\lambda_k u_k \in {\R}^n;
\lambda_1,\ldots,\lambda_k\geq 0\}. 
$$
We say that $\sigma$ is {\it strongly convex} if 
$\sigma\cap(-\sigma)=\{0\}$. 

By regarding a cone as a polyhedron in $\R^n$, 
the definitions of {\it dimension}, {\it face}, 
{\it edge}, {\it facet} for the cone 
are given in the same way as in Section 3.  

The {\it fan} is defined to be a finite collection $\Sigma$ 
of cones in ${\R}^n$ with the following properties:
\begin{itemize}
\item
Each $\sigma\in \Sigma$ 
is a strongly convex rational polyhedral cone;
\item 
If $\sigma\in\Sigma$ and $\tau$ is a face of $\sigma$, then 
$\tau\in \Sigma$;
\item 
If $\sigma,\tau\in\Sigma$, 
then $\sigma\cap\tau$ is a face of each.    
\end{itemize}
For a fan $\Sigma$,
the union $|\Sigma|:=\bigcup_{\sigma\in\Sigma}\sigma$ 
is called the {\it support} of $\Sigma$. 
For $k=0,1,\ldots,n$, we denote by $\Sigma^{(k)}$  
the set of $k$-dimensional cones in $\Sigma$.  
The {\it skeleton} of a cone $\sigma\in\Sigma$ is 
the set of all of its primitive 
integer vectors 
(i.e., with components relatively prime in $\Z_+$)
in the edges of $\sigma$. 
It is clear that the skeleton of $\sigma$ 
generates $\sigma$ itself. 
Thus, the set of skeletons of the cones 
belonging to $\Sigma^{(k)}$ is also expressed 
by the same symbol $\Sigma^{(k)}$.

\subsection{Simplicial subdivision}
%

We denote by $(\R^n)^*$ the dual space of $\R^n$ 
with respect to the standard inner product. 
For $a=(a_1,\ldots,a_n)\in(\R^n)^*$, 
define
\begin{equation}
l(a)=\min\left\{
 \langle a,\alpha \rangle ; \alpha\in\Gamma_{+}(f)
\right\}
\label{eqn:4.1}
\end{equation}
and 
$\gamma(a)=
\{\alpha\in\Gamma_+(f); \langle a,\alpha \rangle =l(a)\}
(=\Gamma_+(f)\cap H(a,l(a)))$.
We introduce an equivalence relation $\sim$ 
in $(\R^n)^*$ by $a\sim a'$ 
if and only if $\gamma(a)=\gamma(a')$. 
For any $k$-dimensional face $\gamma$ of $\Gamma_+(f)$, 
there is an equivalence class $\gamma^*$ which is defined by 
$$
\gamma^*=
\{a\in (\R^n)^*;\gamma(a)=\gamma, \mbox{ and $a_j\geq 0$ 
for $j=1,\ldots,n$}\}.
$$ 
It is easy to see that the closure of 
$\gamma^*$ is an $(n-k)$-dimensional strongly convex rational 
polyhedral cone in $(\R^n)^*$. 
Moreover,  
the collection of the closures of 
$\gamma^*$ gives a fan $\Sigma_0$.  
Note that $|\Sigma_0|=\R_+^n$.


It is known 
that there exists a {\it simplicial subdivision} $\Sigma$
of $\Sigma_0$,  
that is, $\Sigma$ is a fan satisfying the following properties:
\begin{itemize}
\item 
The fans $\Sigma_0$ and $\Sigma$ have the same support; 
\item 
Each cone of $\Sigma$ lies in some cone of $\Sigma_0$; 
\item 
The skeleton of any cone belonging to $\Sigma$ can be completed 
to a base of the lattice dual to $\Z^n$.
\end{itemize}


\subsection{Construction of toric varieties}

Fix a simplicial subdivision $\Sigma$ of $\Sigma_0$. 
For $n$-dimensional cone $\sigma\in\Sigma$, 
let 
$a^1(\sigma),\ldots,a^n(\sigma)$ be the skeleton of 
$\sigma$, ordered once and for all. 
Here, we set the coordinates of the vector $a^j(\sigma)$ as 
$$
a^j(\sigma)=(a^j_1(\sigma),\ldots,a^j_n(\sigma)).
$$
With every such cone $\sigma$, we
associate a copy of $\C^n$ which is denoted by 
$\C^n(\sigma)$.
We denote by
$ 
\pi(\sigma):\C^n(\sigma) \to \C^n
$
the map defined by  
$\pi(\sigma)(y_1,\ldots,y_n)=
(x_1,\ldots,x_n)$ with
\begin{equation}
x_j= y_1^{a_j^1(\sigma)}\cdots y_n^{a_j^n(\sigma)}, \quad\quad 
j=1,\ldots,n.
\label{eqn:4.2} 
\end{equation}
Let $X_{\Sigma}$ be the union of $\C^{n}(\sigma)$ for $\sigma$
which are glued along the images of $\pi(\sigma)$. 
Indeed, for any $n$-dimensional cones $\sigma,\sigma'\in\Sigma$, 
two copies $\C^n(\sigma)$ and $\C^n(\sigma')$ can be
identified with respect to a rational mapping:
$\pi^{-1}(\sigma')\circ \pi(\sigma):
\C^n(\sigma)\to \C^n(\sigma')$ 
(i.e. $x\in\C^n(\sigma)$ and $x'\in\C^n(\sigma')$ will coalesce 
if $\pi^{-1}(\sigma')\circ\pi(\sigma):x\mapsto x'$).
Then it is known that
\begin{itemize}
\item
$X_{\Sigma}$ is an 
$n$-dimensional complex algebraic manifold; 
\item
The map 
$\pi:X_{\Sigma}\to\C^n$
defined on each $\C^n(\sigma)$ as 
$\pi(\sigma):\C^n(\sigma)\to\C^n$ is proper. 
\end{itemize}
The manifold $X_{\Sigma}$ is called the 
{\it toric variety} associated with $\Sigma$.            
The transition functions between local maps of the manifold 
$X_{\Sigma}$ are real on the real part of the manifold 
$X_{\Sigma}$
which will be denoted by $Y_{\Sigma}$. 
The restriction of the projection $\pi$ to $Y_{\Sigma}$ is also
denoted by $\pi$. 
Then we have

\begin{itemize}
\item
$Y_{\Sigma}$ is an $n$-dimensional 
real algebraic manifold;
\item
The map 
$\pi:Y_{\Sigma}\to\R^n$
defined on each $\R^n(\sigma)$ as 
$\pi(\sigma):\R^n(\sigma)\to\R^n$ is proper. 
\end{itemize}

\begin{note}
The map $\pi(\sigma)$ plays an important role 
in our analysis and   
the following kind of computation often appears: 
Let $p=(p_1,\ldots,p_n)\in \R^n$, then  
(\ref{eqn:4.2}) implies 
\begin{eqnarray*}
&&x^p=(\pi(\sigma)(y))^p=%
(y_1^{a_1^1(\sigma)}\cdots y_n^{a_1^n(\sigma)})^{p_1}\cdots
(y_1^{a_n^1(\sigma)}\cdots y_n^{a_n^n(\sigma)})^{p_n}\\
&&\quad=y_1^{ \langle a^1(\sigma),p \rangle }\cdots y_n^{ \langle a^n(\sigma),p \rangle }.
\end{eqnarray*}
\end{note}

\subsection{Resolution of singularities}

For $I\subset \{1,\ldots,n\}$, 
define the set $T_I$ in $\R^n$ by 
\begin{equation}
T_I=\{y\in\R^n; 
y_j=0 
\mbox{ for $j\in I$, }\,
y_j\neq 0 
\mbox{ for $j\not\in I$}\}.  
\label{eqn:4.3}
\end{equation}

The following proposition shows that
$\pi:Y_{\Sigma}\to\R^n$ is a real resolution of the singularity 
of the critical point of a real analytic function 
satisfying the nondegenerate property. 

\begin{proposition}[\cite{var76}, Lemma 2.13, Lemma 2.15]
Suppose that  
$f$ is a real analytic function in a neighborhood 
$U$ of the origin.    
Then we have the following.  
\begin{enumerate}
\item
There exists a real analytic function $f_{\sigma}$ 
defined on the set $\pi(\sigma)^{-1}(U)$ such that
$f_{\sigma}(0)\neq 0$ and 
\begin{equation}
 (f\circ \pi(\sigma))(y_1,\ldots,y_n)
=y_1^{l(a^1(\sigma))}\cdots y_n^{l(a^n(\sigma))} 
f_{\sigma}(y_1,\ldots,y_n).
\label{eqn:4.4}
\end{equation} 
\item 
The Jacobian of the mapping $\pi(\sigma)$ 
is equal to 
\begin{equation}
J_{\pi(\sigma)}(y)
=\pm  
y_1^{ \langle a^1(\sigma) \rangle -1}\cdots 
y_n^{ \langle a^n(\sigma) \rangle -1}.
\label{eqn:4.5}
\end{equation}
\item 
The set of the points in $\R^n$ in which $\pi(\sigma)$ 
is not an isomorphism is a union of coordinate planes. 
\end{enumerate}

Moreover,
if $f$ is nondegenerate over $\R$ 
with respect to $\Gamma_+(f)$ and    
$\pi(\sigma)(T_I)=0$,  
then the set $\{y\in T_I;f_{\sigma}(y)=0\}$ is 
nonsingular, that is, 
the gradient of the restriction of the function 
$f_{\sigma}$ to $T_I$ does not vanish at the points of the set  
$\{y\in T_I;f_{\sigma}(y)=0\}$. 
\end{proposition}


\section{Poles of local zeta functions}

Throughout this section, 
the functions $f$, $\varphi$, $\chi$ always satisfy 
the conditions (A), (B), (C) 
in the beginning of Section 2.2. 

The purpose of this section is 
to investigate the properties of 
poles of the functions:
\begin{equation}
Z_{+}(s)=\int_{\R^n} f(x)_{+}^s \varphi(x)\chi(x)dx, \quad 
Z_{-}(s)=\int_{\R^n} f(x)_{-}^s \varphi(x)\chi(x)dx, 
\label{eqn:5.1}
\end{equation}
where $f(x)_+=\max\{f(x),0\}$ and 
$f(x)_-=\max\{-f(x),0\}$ and the {\it local zeta function}:
\begin{equation}
Z(s)=\int_{\R^n} |f(x)|^s \varphi(x)\chi(x)dx.
\label{eqn:5.2}
\end{equation}
From the properties of $Z_+(s)$ and $Z_-(s)$, 
we can easily obtain analogous properties 
of $Z(s)$  
by using the relationship: $Z(s)=Z_+(s)+Z_-(s)$.

It is easy to see that the above functions 
are holomorphic functions in the region 
${\rm Re} (s)>0$.  
Moreover, it is known (see \cite{agv88},\cite{kan81}, etc.) that 
if the support of $\chi$ is sufficiently small, 
then these functions can be analytically 
continued to the complex plane as 
meromorphic functions and their poles belong 
to finitely many arithmetic progressions 
constructed from negative rational numbers.
More precisely, Varchenko \cite{var76} describes
the positions of the candidate poles and their orders 
by using the toric resolution
constructed in Section~4. 
In this section, we give more accurate results 
in the case that $\varphi$ has a zero 
at the origin. 

\subsection{The monomial case}
First, 
let us consider the case 
that the function $\varphi$ is a monomial, i.e., 
$\varphi(x)=x^p=x_1^{p_1}\cdots x_n^{p_n}$ with 
$p=(p_1,\ldots,p_n)\in\Z_+^n$. 
Fedorjuk \cite{fed59} was the first to consider 
this kind of issue in two-dimensional case. 
Moreover, 
there have been closely related studies to ours in 
\cite{ds89},\cite{ds92},\cite{dls97},\cite{dns05}, 
which contain other interesting results.
 
\begin{theorem}
Suppose that 
{\rm (i)} 
$f$ is nondegenerate over $\R$ 
with respect to its Newton polyhedron and 
{\rm (ii)} 
$\varphi(x)=x^p$ with $p\in\Z_+^n$. 
If the support of $\chi$ is
contained in a sufficiently small neighborhood 
of the origin, 
then the poles of the functions  
$Z_{+}(s)$, $Z_-(s)$ and $Z(s)$ 
are contained in the set 
$$\left\{
-\frac{ \langle a,p+\1 \rangle +\nu}{l(a)}
;\,\, \nu\in\Z_+,\,\, a\in\tilde{\Sigma}^{(1)}
\right\}
\cup (-\N),$$
where $l(a)$ is as in $(\ref{eqn:4.1})$ and 
$\tilde{\Sigma}^{(1)}=\{a\in\Sigma^{(1)};l(a)>0\}$. 
\end{theorem}
In Remark 5.3 after the proof, 
we will explain in more detail the reason why the set 
$(-\N)$ is necessary to express the poles. 

\begin{proof}
Let $\Sigma_0$ be the fan constructed 
from the Newton polyhedron 
of $f$. 
Fix a simplicial subdivision $\Sigma$ of $\Sigma_0$ and 
let $(Y_{\Sigma},\pi)$ be the real resolution 
associated with $\Sigma$ as in Section 4. 

By using the mapping $x=\pi(y)$, 
$Z_{+}(s)$ and $Z_-(s)$ are expressed as 
\begin{eqnarray*}
&&
Z_{\pm}(s)=\int_{\R^n}f(x)_{\pm}^s x^p\chi(x)dx \\
&&
\quad 
=\int_{Y_{\Sigma}} 
((f\circ\pi)(y))_{\pm}^s (\pi(y))^p (\chi\circ\pi)(y) 
|J_{\pi}(y)|dy, 
\end{eqnarray*}
where $dy$ is a volume element in $Y_{\Sigma}$,   
$J_{\pi}(y)$ is the Jacobian of the mapping $\pi$. 
It is easy to see that 
there exists a set of $C^{\infty}_0$ functions 
$\{\chi_{\sigma}:Y_{\Sigma} \to\R_+; \sigma\in\Sigma^{(n)}\}$ 
satisfying 
the following properties:
\begin{itemize}
\item 
For each $\sigma\in\Sigma^{(n)}$, 
the support of the function $\chi_{\sigma}$ is contained 
in $\R^n(\sigma)$ and 
$\chi_{\sigma}$ identically equals one 
in some neighborhood of the origin. 
\item 
$\sum_{\sigma\in\Sigma^{(n)}}\chi_{\sigma}\equiv 1$ 
on the support of 
$\chi\circ\pi$.  
\end{itemize}
Applying Proposition~4.2, we have 
$$
Z_{\pm}(s)=\sum_{\sigma\in\Sigma^{(n)}} 
Z_{\pm}^{(\sigma)}(s)
$$ 
with  
\begin{equation}
\begin{split}
&Z_{\pm}^{(\sigma)}(s)
=\int_{\R^n} ((f\circ\pi(\sigma))(y))_{\pm}^s 
(\pi(\sigma)(y))^p(\chi\circ\pi(\sigma))(y) 
\chi_{\sigma}(y)|J_{\pi(\sigma)}(y)|dy \\
&
\quad 
=\int_{\R^n} 
\left(
\prod_{j=1}^n y_j^{l(a^j(\sigma))}f_{\sigma}(y)
\right)_{\pm}^s 
\left(
\prod_{j=1}^n y_j^{\langle a^j(\sigma),p \rangle}
\right) 
\left|
\prod_{j=1}^n y_j^{ \langle a^j(\sigma) \rangle -1}
\right|\tilde{\chi}_{\sigma}(y)dy, 
\end{split}
\label{eqn:5.3}
\end{equation}
where 
$\tilde{\chi}_{\sigma}(y)=
(\chi\circ\pi(\sigma))(y)\chi_{\sigma}(y)$.

Now, consider the functions $Z^{(\sigma)}_{\pm}(s)$ 
for $\sigma\in\Sigma^{(n)}$. 
We easily see the existence of finite sets of 
$C^{\infty}_0$ functions 
$\{\psi_k:\R^n\to\R_+\}$ and 
$\{\eta_l:\R^n\to\R_+\}$ satisfying the following conditions. 
\begin{itemize}
\item 
The supports of $\psi_k$ and $\eta_l$ are sufficiently small and 
$\sum_k \psi_k + \sum_l \eta_l \equiv 1$ 
on the support of $\tilde{\chi}_{\sigma}$.
\item 
For each $k$, 
$f_{\sigma}$ is always positive or negative on the support of $\psi_k$. 
\item 
For each $l$, the support of $\eta_l$ intersects the set 
$\{y\in {\rm Supp}(\tilde{\chi}_{\sigma});f_{\sigma}(y)=0\}$
\item
The union of the support of $\eta_l$ for all $l$ contains the set 
$\{y\in {\rm Supp}(\tilde{\chi}_{\sigma});f_{\sigma}(y)=0\}$
\end{itemize}

Using the functions $\psi_k$ and $\eta_l$, we have
\begin{equation}
Z^{(\sigma)}_{\pm}(s)=
\sum_k I^{(k)}_{\sigma,\pm}(s)+
\sum_l J^{(l)}_{\sigma,\pm}(s),
\label{eqn:5.4}
\end{equation} 
with
\begin{eqnarray*}
&&
I^{(k)}_{\sigma,\pm}(s)=\int_{\R^n} 
\left(
\prod_{j=1}^n 
y_j^{l(a^j(\sigma))}f_{\sigma}(y)
\right)_{\pm}^s 
\left(
\prod_{j=1}^n y_j^{\langle a^j(\sigma),p \rangle}
\right) 
\left|
\prod_{j=1}^n y_j^{ \langle a^j(\sigma) \rangle -1}
\right|\tilde{\psi}_k(y)dy, 
\\
&&
J^{(l)}_{\sigma,\pm}(s)
=\int_{\R^n} 
\left(
\prod_{j=1}^n y_j^{l(a^j(\sigma))}f_{\sigma}(y)
\right)_{\pm}^s 
\left(
\prod_{j=1}^n y_j^{\langle a^j(\sigma),p \rangle}
\right) 
\left|
\prod_{j=1}^n y_j^{ \langle a^j(\sigma) \rangle -1}
\right|\tilde{\eta}_l(y)dy, 
\end{eqnarray*}
where $\tilde{\psi}_k(y)=\tilde{\chi}_{\sigma}(y)\psi_k(y)$ and 
$\tilde{\eta}_l(y)=\tilde{\chi}_{\sigma}(y)\eta_l(y)$. 
If the set 
$\{y\in {\rm Supp}(\tilde{\chi}_{\sigma});f_{\sigma}(y)=0\}$
is empty, then the functions $J^{(l)}_{\sigma,\pm}(s)$ 
do not appear. 

First, 
consider the functions $I^{(k)}_{\sigma,\pm}(s)$. 
Set $\delta(+)=0$ and $\delta(-)=1$. 
For $\epsilon=(\epsilon_1,\ldots,\epsilon_n)\in\{+,-\}^n$, 
let $\delta(\epsilon)=
(\delta(\epsilon_1),\ldots,\delta(\epsilon_n))\in\{0,1\}^n$.
A straightforward computation gives
\begin{equation}
I^{(k)}_{\sigma,\pm}(s)=
\sum_{\epsilon\in E(\pm \a_k, \sigma)}
I^{(\epsilon)}_{\sigma,k}(s) 
\label{eqn:Ikspm}
\end{equation}
with 
\begin{equation*}
\begin{split}
I^{(\epsilon)}_{\sigma,k}(s)=
\int_{\R^n} 
\left(
\prod_{j=1}^n 
(y_j)_{\epsilon_j}^{l(a^j(\sigma))s}
\right)
\left(
\prod_{j=1}^n y_j^{\langle a^j(\sigma),p \rangle}
\right) 
\left|
\prod_{j=1}^n y_j^{ \langle a^j(\sigma) \rangle -1}
\right|
|f_{\sigma}(y)|^s
\tilde{\psi}_k(y)dy, 
\end{split}
\label{eqn:}
\end{equation*}
where 
$\alpha_k$ is the sign of $f_{\sigma}$ 
on the support of $\tilde{\psi}_k$ and 
\begin{equation*}
\begin{split}
&E(+,\sigma) \mbox{ (resp. $E(-,\sigma)$)} \\
&\quad\quad
=\left\{
\epsilon\in\{+,-\}^n; 
\sum_{j=1}^n l(a^j(\sigma))\delta(\epsilon_j)
\mbox{ is even (resp. odd)} 
\right\}. 
\end{split}
\end{equation*}
We remark that $E(+,\sigma)\cup E(-,\sigma)=\{+,-\}^n$ 
and that $E(-,\sigma)$ is possibly empty. 
Moreover, we have 
\begin{equation}
I_{\sigma,k}^{(\epsilon)}(s)
=
(-1)^{g_{\sigma,p}(\epsilon)}
\int_{\R^n}
\left(
\prod_{j=1}^n 
(y_j)_{\epsilon_j}^
{l(a^j(\sigma))s+\langle a^j(\sigma),p+\1\rangle-1}
\right)
|f_{\sigma}(y)|^s\tilde{\psi}_{k}(y)dy,
\label{eqn:5.5}
\end{equation}
where 
\begin{equation}
g_{\sigma,p}(\epsilon)=
\sum_{i,j=1}^n
\delta(\epsilon_j) \cdot a_i^j(\sigma) \cdot p_i.
\label{eqn:gsp}
\end{equation}

The following lemma is useful for analyzing 
the poles of integrals of the above form. 
\begin{lemma}[\cite{gs64},\cite{agv88}]
Let $\psi(y_1,\ldots,y_n;\mu)$ be a $C^{\infty}_0$ function 
of $y$ on $\R^n$ that is an entire function 
of the parameter $\mu\in\C$. 
Then the function 
$$
L(\tau_1,\ldots,\tau_n;\mu)=\int_{\R^n}
\left(\prod_{j=1}^n (y_j)_{\epsilon_j}^{\tau_j}\right)
\psi(y_1,\ldots,y_n;\mu) dy_1\cdots dy_n, 
$$
where $\epsilon_j$ is $+$ or $-$, 
can be analytically continued at all the values of 
$\tau_1,\ldots,\tau_n$ and $\mu$ as a meromorphic function. 
Moreover all its poles are simple and lie on 
$\tau_j=-1,-2,\ldots$ for $j=1,\ldots,n$. 
\end{lemma}
\begin{proof}
This is easily obtained by the integration by parts 
(see \cite{gs64},\cite{agv88}).
\end{proof}
Applying Lemma~5.2 to (\ref{eqn:5.5}), 
we see that  
the poles of $I_{\sigma,k}^{(\epsilon)}(s)$ 
are contained in 
the set 
\begin{equation}
\left\{
-\frac{ \langle a^j(\sigma),p+\1 \rangle +\nu}{l(a^j(\sigma))};
\nu\in\Z_+, j\in B({\sigma})
\right\},
\label{eqn:5.6}
\end{equation}
where 
\begin{equation}
B(\sigma):=\{j;l(a^j(\sigma))\neq 0\}
\subset\{1,\ldots,n\}.
\label{eqn:5.7}
\end{equation}

Next, consider the functions $J^{(l)}_{\sigma,\pm}(s)$.  
Applying Proposition~4.2 
and changing the integral variables, 
we have 
\begin{equation}
J^{(l)}_{\sigma,\pm}(s)=
\int_{\R^n} 
\left(
y_k
\prod_{j \in B_l(\sigma)} y_j^{l(a^j(\sigma))}
\right)^s_{\pm} 
\left(
\prod_{j \in B_l(\sigma)} y_j^{\langle a^j(\sigma),p \rangle}
\right) 
\left|
\prod_{j \in B_l(\sigma)} y_j^{ \langle a^j(\sigma) \rangle -1}
\right|\hat{\eta}_l(y)dy, 
\label{eqn:5.71}
\end{equation}
where 
$B_l(\sigma)$ is some subset in $\{1,\ldots,n\}$ 
(with $B_l(\sigma)\neq\{1,\ldots,n\}$), 
$k\in\{1,\ldots,n\}\setminus B_l(\sigma)$ and 
$\hat{\eta}_l\in C^{\infty}_0(\R^n)$ 
with $\hat{\eta}_l(0)\neq 0$.
In a similar fashion to the case of 
$I_{\sigma,\pm}^{(k)}(s)$, 
we have
$$
J_{\sigma,\pm}^{(l)}(s)=
\sum_{\epsilon\in \tilde{E}(\pm,\sigma)}
J_{\sigma,l}^{(\epsilon)}(s), 
$$
with
\begin{equation}
J_{\sigma,l}^{(\epsilon)}(s)
=
(-1)^{\tilde{g}_{\sigma,p}(\tilde{\epsilon})}
\int_{\R^n}
\left(
(y_k)_{\epsilon_k}^s
\prod_{j\in B_l(\sigma)}
(y_j)_{\epsilon_j}^
{l(a^j(\sigma))s+\langle a^j(\sigma),p+\1\rangle-1}
\right)
\hat{\eta}_{l}(y)dy,
\label{eqn:5.8}
\end{equation}
where $\epsilon=(\epsilon_k,\tilde{\epsilon})$ with 
$\tilde{\epsilon}=(\epsilon_j)_{j\in B_l}$, 
$\tilde{g}_{\sigma,p}(\tilde{\epsilon})=
\sum_{j\in B_l(\sigma)}\sum_{i=1}^n
\delta(\epsilon_j) \cdot a_i^j(\sigma) \cdot p_i$ 
and
\begin{equation*}
\begin{split}
&\tilde{E}(+,\sigma) \mbox{ (resp. $\tilde{E}(-,\sigma)$)} 
\\
&\quad\quad
=\left\{
\epsilon=(\epsilon_k,\tilde{\epsilon}); 
\delta(\epsilon_k)+\sum_{j\in B_l(\sigma)} 
l(a^j(\sigma))\delta(\epsilon_j) 
\mbox{ is even (resp. odd)} 
\right\}. 
\end{split}
\end{equation*}
We remark that 
$\card\tilde{E}(+,\sigma)
=\card\tilde{E}(-,\sigma)$ and, 
in particular, 
both $\tilde{E}(+,\sigma)$ and 
$\tilde{E}(-,\sigma)$ are nonempty. 

By applying Lemma~5.2 to (\ref{eqn:5.8}), 
the poles of 
$J_{\sigma,l}^{(\epsilon)}(s)$
are contained in 
the set  
\begin{equation}
\left\{
-\frac{ \langle a^j(\sigma),p+\1 \rangle +\nu}{l(a^j(\sigma))};
\nu\in\Z_+, j\in \tilde{B}_l(\sigma) 
\right\}\cup(-\N), 
\label{eqn:5.9}
\end{equation}
where $\tilde{B}_l(\sigma)=
\{j\in B_l(\sigma);l(a^j(\sigma))\neq 0\}$.

Finally, the union of the sets (\ref{eqn:5.6})
and (\ref{eqn:5.9}) for all $\sigma$, 
$\epsilon$
equals the set in the theorem. 
It is easy to show the case of $Z(s)$
by using the relationship:
$Z(s)=Z_+(s)+Z_-(s)$.
\end{proof} 
\begin{remark}
We explain in more detail the reason why 
the set $(-\N)$ in (\ref{eqn:5.9}) is necessary 
to express the poles of 
$J_{\sigma,l}^{(\epsilon)}(s)$, 
namely, each $J_{\sigma,l}^{(\epsilon)}(s)$
possibly has a pole on $(-\N)$.  
From Proposition 4.2, 
the nondegenerate condition of $f$ implies that 
$f_{\sigma}$ is nonsingular at the zero set 
of $f_{\sigma}$. By choosing an appropriate
coordinate system near the zero set of $f_{\sigma}$, 
$f$ can be locally expressed 
by $y_k(\prod_{j\in B_l(\sigma)}y_j^{l(a^j(\sigma))})$
as in (\ref{eqn:5.71}). 
Moreover, 
the existence of  $(y_k)^s_{\epsilon_k}$ in (\ref{eqn:5.8})
induces the poles on $(-\N)$  
by Lemma 5.2.
\end{remark}

For $p\in\Z_+^n$, we define
\begin{equation}
\beta(p)=\max\left\{
-\dfrac{ \langle a,p+\1 \rangle }{l(a)}; 
a\in\tilde{\Sigma}^{(1)} 
\right\}.
\label{eqn:5.10}
\end{equation}
If $s=\beta(p)$ is a pole of $Z_{\pm}(s)$, $Z(s)$,
then we denote by $\eta_{\pm}(p)$, $\hat{\eta}(p)$ 
the order of its pole, respectively.
For $\sigma\in\Sigma^{(n)}$, let  
$$
A_p(\sigma)=\left\{
j\in B(\sigma);
\beta(p)=-\frac{ \langle a^j(\sigma),p+\1 \rangle }{l(a^j(\sigma))} 
\right\}\subset\{1,\ldots,n\}.
$$

The following proposition shows the relationship between 
``the values of 
$\beta(p)$, $\eta_{\pm}(p)$, $\hat{\eta}(p)$'' and 
``the geometrical conditions of 
$\Gamma_+(f)$ and the point $p$''.
\begin{proposition}
Let $q=(q_1,\ldots,q_n)$ be the point of the 
intersection of $\partial\Gamma_+(f)$ 
with the line joining the origin 
and the point $p+\1=(p_1+1,\ldots,p_n+1)$. 
Then 
\begin{eqnarray*}
&& 
-\beta(p)=\frac{p_1+1}{q_1}=\cdots=\frac{p_n+1}{q_n}
=\frac{ \langle p \rangle +n}{ \langle q \rangle }=\frac{1}{d(f,x^p)},\\
&&
\eta_{\pm}(p),\hat{\eta}(p)\leq
\begin{cases}
\rho_f(q)& \quad \mbox{if $1/d(f,x^p)$ 
is not an integer}, \\
\min\{\rho_f(q)+1, n\}&
\quad \mbox{otherwise},
\end{cases}
\end{eqnarray*}
where $\rho_f$ and $d(\cdot,\cdot)$ are as 
in Section~$2.1$.
Note that 
$m(f,x^p)=\rho_f(q)=\rho_f(d(f,x^p)(p+\1))$.
\end{proposition}
\begin{remark}
In the case when $n=2$ or $3$, 
$\rho_f(q)$ is equal to $\min\{\hat{m}_p,n\}$, where  
$\hat{m}_p$ is the number of the $(n-1)$-dimensional 
faces of $\Gamma_+(f)$ containing the point $q$. 
This, however, does not generally hold for
$n\geq 4$. 
\end{remark}

\begin{proof}

For $a \in \Sigma^{(1)}$, 
we denote by $q(a)$ 
the point of the intersection of the hyperplane 
$H(a,l(a))$ with the line $\{t\cdot (p+\1);t\in\R\}$, 
where $H(\cdot,\cdot)$ is as in (\ref{eqn:3.1}). 
Then it is easy to see 
\begin{equation}
q(a)=\frac{l(a)}{ \langle a,p+\1 \rangle }\cdot (p+\1). 
\label{eqn:5.11}
\end{equation}
From (\ref{eqn:5.11}),  
the condition that 
$-\dfrac{ \langle a,p+\1 \rangle }{l(a)}$ takes the maximum
is equivalent to 
the geometrical condition that 
$q(a)$ is as far as possible from the origin. 
To be more precise, we have the following equivalences: 
For $a\in\tilde{\Sigma}^{(1)}$,  
\begin{equation}
\beta(p)=-\frac{ \langle a,p+\1 \rangle }{l(a)} \, \Longleftrightarrow \,
q=q(a) \, \Longleftrightarrow \,
q\in H(a,l(a)). 
\label{eqn:5.12}
\end{equation}
From (\ref{eqn:5.11}) and (\ref{eqn:5.12}), we have
$-\beta(p)=(p_1+1)/q_1=
\cdots=(p_n+1)/q_n=( \langle p \rangle +n)/ \langle q \rangle $.
From the definition of $d(\cdot,\cdot)$, 
the above value equals $1/d(f,x^p)$. 
Note that $q=d(f,x^p)(p+\1)$.

Next, consider the orders of the poles
of $Z_{\pm}(s), Z(s)$ at $s=\beta(p)$.
From the proof of Theorem~5.1, 
it suffices to analyze the poles of 
$I_{\sigma,k}^{(\epsilon)}(s)$,
$J_{\sigma,l}^{(\epsilon)}(s)$. 
Applying Lemma~5.2 to the integrals 
(\ref{eqn:5.5}),(\ref{eqn:5.8}), 
we see the upper bounds of orders of the poles 
at $s=\beta(p)$ 
of these functions as follows.
\begin{center}
\begin{tabular}{l|l} \hline
\textit{$I_{\sigma,k}^{(\epsilon)}(s)$} & 
       \quad $\card A_p(\sigma)$ \\ \hline
\textit{$J_{\sigma,l}^{(\epsilon)}(s)$} & 
       $\min\{\card A_p(\sigma),n-1\}$ if $\beta(p)\not\in (-\N)$\\
             & $\min\{\card A_p(\sigma)+1,n\}$ if $\beta(p) \in (-\N)$ 
\\ \hline
\end{tabular}
\end{center}
From these estimates of orders, 
in order to obtain the estimates in 
the proposition, it suffices to 
show 
$\rho_f(q)=
\max\left\{\card A_p(\sigma)
;\sigma\in\Sigma^{(n)}
\right\}.$
From the definition of $A_p(\sigma)$
and (\ref{eqn:5.12}), we have 
\begin{eqnarray*}
&&\card A_p(\sigma)
= \card\{j ; q\in H(a^j(\sigma),l(a^j(\sigma)))\} \\
&& \quad=\card\{j ; 
\gamma\subset H(a^j(\sigma),l(a^j(\sigma)))\}, 
\end{eqnarray*}
where $\gamma$ is the face of $\Gamma_+(f)$ whose relative interior
contains the point $q$ 
(i.e., $q\in(\gamma\setminus\partial\gamma)$.)
From the definition of $\rho_f$, 
the codimension of $\gamma$ is $\rho_f(q)$. 
Since $a^1(\sigma),\ldots,a^n(\sigma)$ are linearly independent 
for each $\sigma\in\Sigma^{(n)}$, 
$\card\{j ; \gamma\subset H(a^j(\sigma),l(a^j(\sigma)))\}$ 
is not larger than $\rho_f(q)$ for any $\sigma\in\Sigma^{(n)}$. 
On the other hand, the closure of $\gamma^*$ is a cone belonging to
the fan $\Sigma_0$ constructed from $\Gamma_+(f)$ 
(see Section~4.2) and the dimension of this cone is $\rho_f(q)$. 
There exists an $n$-dimensional cone 
$\hat{\sigma}$ in a simplicial subdivision 
$\Sigma$ of $\Sigma_0$ whose $\rho_f(q)$-dimensional face 
is contained in the closure of $\gamma^*$. 
This means 
$\card\{j ; \gamma\subset H(a^j(\hat{\sigma}),l(a^j(\hat{\sigma})))\}
=\rho_f(q)$. 
Hence, we see $\rho_f(q)=
\max\left\{\card A_p(\sigma)
;\sigma\in\Sigma^{(n)}
\right\}.$

Lastly, it follows from $\Gamma(x^p,f)=\Gamma_0=\{p\}$ that
$m(f,x^p)=\rho_f(d(f,x^p)(p+\1))=\rho_f(q)$.
\end{proof}

Next, let us consider the coefficients of the Laurent
expansions of $Z_{+}(s)$ and $Z_-(s)$ at the poles. 
The following lemma is useful 
for computing the coefficients explicitly. 
\begin{lemma}
Let $\psi$ be a $C^{\infty}$ function on $\R$ and $k\in\N$. 
Then 
$$
\lim_{s\to -k} (s+k)
\int_{-\infty}^{\infty}
y_{\pm}^s \psi(y)dy 
=\frac{(\pm 1)^{k-1}}{(k-1)!} \psi^{(k-1)}(0). 
$$
In particular, we have
$$
\lim_{s\to -1} (s+1)
\int_{-\infty}^{\infty}
y_{\pm}^s \psi(y)dy 
=\psi(0). 
$$
\end{lemma}
\begin{proof}
The above formula is easily obtained 
by the integration by parts. 
\end{proof}
When $d(f,x^p)>1$, 
we compute the coefficients of 
$(s-\beta(p))^{-m(f,x^p)}$ in the Laurent expansions of 
$Z_{\pm}(s)$, $Z(s)$. 
Let 
$$ 
C_{\pm}=\lim_{s\to\beta(p)} 
(s-\beta(p))^{m(f,x^p)} 
Z_{\pm}(s), 
\quad
C=\lim_{s\to\beta(p)} 
(s-\beta(p))^{m(f,x^p)} 
Z(s),
$$
respectively.
\begin{theorem}
Suppose that 
{\rm (i)} 
$f$ is nondegenerate over $\R$ 
with respect to its Newton polyhedron, 
{\rm (ii)} 
$\varphi(x)=x^p$, 
where every component of $p\in \Z_+^n$ is even, and 
{\rm (iii)} 
$d(f,x^p)>1$. 
If the support of $\chi$ is
contained in a sufficiently small neighborhood 
of the origin,  
then 
$C_+$ and $C_-$ are nonnegative and 
$C=C_{+}+C_{-}$ is positive. 
\end{theorem}
\begin{proof}
Let 
$$
{\Sigma}_p^{(n)}=
\{\sigma\in\Sigma^{(n)};\card A_p(\sigma)=m(f,x^p)\}.
$$ 

First, we consider the case when $m(f,x^p)<n$. 
For $\sigma\in{\Sigma}_p^{(n)}$, 
considering the equations (\ref{eqn:5.4}) and 
applying Lemma~5.6 to (\ref{eqn:5.5}), (\ref{eqn:5.8}) 
with respect to each $y_j$ for $j\in A_p(\sigma)$, 
we have
\begin{equation}
\lim_{s\to\beta(p)} 
(s-\beta(p))^{m(f,x^p)} 
Z^{(\sigma)}_{\pm}(s)=
\sum_{k} G_{\pm}^{(k)}(\sigma) +
\sum_{l} H_{\pm}^{(l)}(\sigma),
\label{eqn:5.13}
\end{equation}
with 
\begin{equation}
\begin{split}
& 
G_{\pm}^{(k)}(\sigma)=
\sum_{\epsilon \in E(\pm{\alpha_k},\sigma)}
\frac{(-1)^{g_{\sigma,p}(\epsilon)}}
{\prod_{j\in A_p(\sigma)}l(a^j(\sigma))}\cdot \\
&
\quad\quad\quad
\int_{D(\sigma)}
\left(
\prod_{j\not\in A_p(\sigma)}(y_j)_{\epsilon_j}^{
l(a^j(\sigma))\beta(p)+\langle a^j(\sigma),p+\1\rangle-1
}
\right)
|f_{\sigma}(\hat{y})|^{\beta(p)}
\tilde{\psi}_{k}(\hat{y}) d\hat{y}, 
\end{split}
\label{eqn:5.14}
\end{equation}
\begin{equation}
\begin{split}
&
H_{\pm}^{(l)}(\sigma)=
\sum_{\tilde{\epsilon} \in \tilde{E}(\pm,\sigma)}
\frac{(-1)^{\tilde{g}_{\sigma,p}(\tilde{\epsilon})}}
{\prod_{j\in A_p(\sigma)}l(a^j(\sigma))}\cdot \\
&
\quad\quad\quad
\int_{D(\sigma)}
(y_k)_{\epsilon_k}^{\beta(p)}
\left(
\prod_{j\in B_l(\sigma)\setminus A_p(\sigma)}
(y_j)_{\epsilon_j}^{
l(a^j(\sigma))\beta(p)+\langle a^j(\sigma),p+\1\rangle-1
}
\right)
\hat{\eta}_{l}(\hat{y}) d\hat{y},
\end{split}
\label{eqn:5.15}
\end{equation}
where  
the summations in (\ref{eqn:5.13}) are taken for all $k$,$l$ 
satisfying 
$T_{A_p(\sigma)}\cap {\rm Supp}(\psi_k)\neq \emptyset$ 
and $A_p(\sigma)\subset B_l(\sigma)$, 
$\hat{y}$ is defined by 
$\hat{y}_j=0$ for $j\in A_p(\sigma)$,  
$\hat{y}_j=y_j$ for $j\not\in A_p(\sigma)$, 
$d\hat{y}=\prod_{j\not\in A_p(\sigma)} dy_j$, 
$D(\sigma)=\{
y\in\R^n;y_j=0 \mbox{ for $j\in A_p(\sigma)$}
\}
(\approx
\R^{n-m(f,x^p)})$ and 
the other symbols are the same 
as in (\ref{eqn:5.5}), (\ref{eqn:5.8}). 
Note that the integrals in (\ref{eqn:5.15}) are convergent and 
interpreted as improper integrals. 
In (\ref{eqn:5.14}), (\ref{eqn:5.15}), 
we deform the cut-off functions $\psi_k$ and $\eta_l$ 
as the volume of the support of $\eta_l$ tends to
zero for all $l$.
Then it is easy to see that 
the limit of $H_{\pm}^{(l)}(\sigma)$ is zero, while 
that of $\sum_k G_{\pm}^{(k)}(\sigma)$ is 
$G_{+,\pm}(\sigma)+G_{-,\mp}(\sigma)$, respectively, with 
\begin{equation}
\begin{split}
&G_{u,v}(\sigma)=
\sum_{\epsilon \in E(u,\sigma)}
\frac{(-1)^{g_{\sigma,p}(\epsilon)}}
{\prod_{j\in A_p(\sigma)}l(a^j(\sigma))}\cdot \\
&
\quad\quad
\int_{D_{v}(\sigma)}
\left(
\prod_{j\not\in A_p(\sigma)}(y_j)_{\epsilon_j}^{
l(a^j(\sigma))\beta(p)+\langle a^j(\sigma),p+\1\rangle-1
}
\right)
|f_{\sigma}(\hat{y})|^{\beta(p)}
\tilde{\chi}_{\sigma}(\hat{y}) d\hat{y},
\end{split}
\label{eqn:5.16}
\end{equation}
where $u,v\in \{+,-\}$ and 
$$
D_{v}(\sigma)=\{y\in{\rm Supp}(\tilde{\chi}_{\sigma});
v f_{\sigma}(y)>0 \mbox{ and $y_j=0$ for $j\in A_p(\sigma)$}\}.
$$ 
Note that the above integral is also improper, 
if $f_{\sigma}$ has a zero on $D(\sigma)$. 
As a result, we have 
\begin{equation}
\begin{split}
&C_{\pm}=
\sum_{\sigma\in {\Sigma}_p^{(n)}}
(G_{+,\pm}(\sigma)+G_{-,\mp}(\sigma)),\\
\end{split}
\label{eqn:5.17}
\end{equation}
respectively.

Next, we consider the case when $m(f,x^p)=n$. 
Noticing 
$\Sigma_p^{(n)}=
\{\sigma;A_p(\sigma)=B(\sigma)=\{1,\ldots,n\}\}$, 
we obtain the corresponding coefficients as  
\begin{equation}
C_{\pm}
=
\sum_{\sigma\in {\Sigma}_p^{(n)}}
\sum_{\epsilon \in E(\pm\alpha,\sigma)}
\frac{(-1)^{g_{\sigma,p}(\epsilon)}|f_{\sigma}(0)|^{\beta(p)}}
{\prod_{j=1}^n l(a^j(\sigma))},
\label{eqn:5.16n}
\end{equation}
where $\alpha$ is the sign of $f_{\sigma}(0)$. 

Now, let us assume that 
every component of $p$ is even. 
By the definition (\ref{eqn:gsp}),  
$g_{\sigma,p}(\epsilon)$ is also even.   
From 
(\ref{eqn:5.16}),(\ref{eqn:5.17}),(\ref{eqn:5.16n}), 
we see  
the nonnegativity
of the coefficients $C_+$, $C_-$. 
Moreover, since $E(+,\sigma)$  
is nonempty, the coefficient
$C=C_++C_-$ is positive. 
\end{proof}
The following proposition is concerned with 
the poles of 
$Z_{+}(s)$ and $Z_-(s)$, 
which are induced by the set of zeros of $f_{\sigma}$.
\begin{proposition}
Suppose that 
the conditions {\rm (i)}, {\rm (ii)} in 
Theorem~$5.1$ are satisfied
and 
{\rm (iii)} $d(f,x^p)<1$. 
Let $1,\ldots,k_*$ be all the natural numbers strictly smaller than 
$-\beta(p)=1/d(f,x^p)$. 
If the support of $\chi$ is
contained in a sufficiently small neighborhood 
of the origin, then 
$Z_{+}(s)$ and $Z_{-}(s)$ have 
at $s=-1,\ldots,-k_*$ poles of order not 
higher than $1$ and do not have other poles 
in the region ${\rm Re}(s)>\beta(p)$.    
Moreover, 
let $a_k^{+}$, $a_k^-$ be the residues of 
$Z_{+}(s)$, $Z_-(s)$ at $s=-k$, 
respectively,  
then we have $a_k^+=(-1)^{k-1}a_k^-$ for $k=1,\ldots,k_*$.
\end{proposition}
\begin{proof}
For $j=2,\ldots,n$, let $l_j,m_j$ be positive integers 
such that $m_j/l_j>k_*$.
Let $\epsilon_j=+$ or $-$, $j=2,\ldots,n$, 
be arbitrarily fixed. 
Let $C_k^{\pm}$ be the residues at $s=-k$ of 
the functions
\begin{equation*}
g_{\pm}(s):=
\int_{\R^n}
\left(
(y_1)_{\pm}^s
\prod_{j\in B}
(y_j)_{\epsilon_j}^
{l_j s+m_j-1}
\right)
\eta(y)dy,
\end{equation*}
respectively, 
where $B$ is a subset in $\{2,\ldots,n\}$ and 
$\eta\in C_0^{\infty}(\R^n)$ with 
$\eta(0)\neq 0$. 

By carefully observing 
the analysis of $J_{\sigma,\pm}^{(l)}(s)$ 
in the proof of Theorem~5.1, 
it suffices to show the following. 
\begin{enumerate}
\item[(a)] 
$g_{+}(s)$ and $g_{-}(s)$ have at $s=-1,\ldots,-k_*$ 
poles of order $1$ 
and they do not have other poles in 
Re$(s)>\beta(p)$; 
\item[(b)] 
$C_k^+=(-1)^{k-1}C_k^-$ for $k=1,\ldots,k_*$.  
\end{enumerate}
From Lemma 5.2, (a) is easy to see. 
By using Lemma 5.6, we obtain 
$$
C_k^{\pm}=
\frac{(\pm 1)^{k-1}}{(k-1)!} 
\int_{\R^{n-1}}
\left(
\prod_{j\in B}
(y_j)_{\epsilon_j}^
{-l_j k+m_j -1}
\right)
\frac{\partial^{k-1}\eta}{\partial y_1^{k-1}}
(0,y_2,\ldots,y_n)
dy_2 \cdots dy_n.
$$
This expression implies (b). 
\end{proof}

\begin{remark}
We can easily generalize the results in this subsection 
as follows. 
The same assertions in 
Theorem~5.1, Theorem~5.7 and Proposition~5.8  
can be obtained,  
even if $x^p$ is replaced by $x^p\tilde{\varphi}(x)$ where 
$\tilde{\varphi}\in C^{\infty}(U)$ 
with $\tilde{\varphi}(0)\neq 0$.
Here, in the case of Theorem~5.7, 
when $\tilde{\varphi}(0)<0$,
``positive'' and ``nonnegative'' 
must be changed to 
``negative'' and ``nonpositive'', respectively.
\end{remark}

\subsection{The convenient case}
Next, let us consider the poles of 
$Z_{\pm}(s)$ in (\ref{eqn:5.1}) 
and 
$Z(s)$ in (\ref{eqn:5.2}) 
in the case that $f$ or $\varphi$ is convenient, i.e., 
the associated Newton polyhedron intersects 
all the coordinate axes. 
\begin{theorem}
Suppose that 
{\rm (i)} 
$f$ is nondegenerate over $\R$ 
with respect to its Newton polyhedron and 
{\rm (ii)} 
at least one of the following conditions is satisfied:
\begin{enumerate}
\item[(a)] $f$ is convenient;
\item[(b)] $\varphi$ is convenient; 
\item[(c)] $\varphi$ is  real analytic
on a neighborhood of the origin. 
\end{enumerate}
If the support of $\chi$ is
contained in a sufficiently small neighborhood 
of the origin,  then
the poles of the functions 
$Z_{+}(s)$, $Z_-(s)$ and $Z(s)$ are contained 
in the set 
$$\left\{
-\frac{\tilde{l}(a)+ \langle a \rangle +\nu}{l(a)}
;\,\, \nu\in\Z_+,\,\, a\in\tilde{\Sigma}^{(1)}
\right\}
\cup (-\N),$$
where $l(a)$ is as in $(\ref{eqn:4.1})$, 
$\tilde{l}(a)$ is as in Lemma 5.11, below, 
and $\tilde{\Sigma}^{(1)}$ is as in 
Theorem~$5.1$, and 
$$
\max
\left\{
-\frac{\tilde{l}(a)+ \langle a \rangle }{l(a)}
;\,\,  a\in\tilde{\Sigma}^{(1)}
\right\}
=-\frac{1}{d(f,\varphi)}.
$$

Moreover, for each $Z_+(s)$,$Z_-(s)$ and $Z(s)$, 
if $s=-1/d(f,\varphi)$ is a pole, 
then its order is not larger than 
\begin{equation*}
\begin{cases}
m(f,\varphi)& \quad \mbox{if $1/d(f,\varphi)$ 
is not an integer}, \\
\min\{m(f,\varphi)+1, n\}&
\quad \mbox{otherwise}.
\end{cases}
\end{equation*}
\end{theorem}
\begin{proof}
Let $\Sigma_f$ and $\Sigma_{\varphi}$ 
be the fans constructed from 
the Newton polyhedra of $f$ and $\varphi$, respectively. 
Define $\Sigma_0=\{\sigma\cap\tilde{\sigma};
\sigma\in\Sigma_f,\,\,
\tilde{\sigma}\in\Sigma_{\varphi}\}$. 
Then it is easy to see that $\Sigma_0$ is also a fan. 
Fix a simplicial subdivision $\Sigma$ of $\Sigma_0$
and let $(Y_{\Sigma},\pi)$ be the real resolution 
associated with $\Sigma$ as in Section 4. 

First, let us compute the form of $\varphi\circ\pi$. 
\begin{lemma}
Let $\varphi$ be a $C^{\infty}$ function defined
on a neighborhood of the origin. 
When $\varphi$ is convenient or real analytic near the origin, 
define $\tilde{l}(a)=\min\{ \langle a,\a \rangle ;\a\in\Gamma_+(\varphi)\}$
for $a\in \Z_+^n$. 
Otherwise, define 
$\tilde{l}(a)=\min\{ \langle a,\a \rangle ;\a\in\Gamma_+(\varphi)\}$
for $a\in \N^n$ and 
$\tilde{l}(a)=0$ for 
$a\in \Z_+^n\setminus \N^n$. 
Then, for $\sigma\in\Sigma^{(n)}$, 
$\varphi\circ\pi(\sigma)$ can be expressed as  
\begin{equation}
\varphi(\pi(\sigma)(y))=
\left(
\prod_{j=1}^n y_j^{\tilde{l}(a^j(\sigma))}
\right)
\varphi_{\sigma}(y),
\label{eqn:5.18}
\end{equation}
where 
$\varphi_{\sigma}$ is a 
$C^{\infty}$ function defined on a neighborhood of the origin. 
$($Needless to say, 
if $\varphi$ is real analytic, so is $\varphi_{\sigma}$.$)$  
\end{lemma}
\begin{proof}[Proof of Lemma~5.11]
Let us consider the case that 
$\varphi$ is a $C^{\infty}$ function. 
By using Taylor's formula, 
$\varphi$ can be expressed as (\ref{eqn:3.4})
in Section~3. 
Substituting $x=\pi(\sigma)(y)$ with (\ref{eqn:4.2}) 
into (\ref{eqn:3.4}), 
we have 
\begin{equation*}
\begin{split}
&\varphi(\pi(\sigma)(y))= \\
&\sum_{\a\in S_{\varphi}\cap U_N}
c_{\a}\left(
\prod_{j=1}^n y_j^{ \langle a^j(\sigma),\a \rangle }
\right)+
\sum_{p\in\Z_+^n, \langle p \rangle =N}
\left(
\prod_{j=1}^n y_j^{ \langle a^j(\sigma),p \rangle }
\right)
\varphi_{p}(\pi(\sigma)(y)).
\end{split}
\end{equation*}
Here, take a sufficiently large $N\in\N$ such that 
the union of the hypersurfaces $H(a,\tilde{l}(a))\cap \R_+^n$ 
for all $a\in\Sigma^{(1)}\cap\N^n$ is contained in 
the set $U_N=\{\a\in\R_+^n; \langle \a \rangle \leq N\}$.
If $a\in\Sigma^{(1)}\cap \N^n$, then 
we see that 
$ \langle a,\alpha \rangle \geq \tilde{l}(a)$
for $\alpha\in\Gamma_+(\varphi)$ and 
$ \langle a,p \rangle \geq \tilde{l}(a)$ 
for $p\in\Z_+^n$ with $ \langle p \rangle =N$.
Therefore, we can get 
\begin{equation*}
\varphi(\pi(\sigma)(y))=
\left(
\prod_{j\in C(\sigma)} y_j^{\tilde{l}(a^j(\sigma))}
\right)
\varphi_{\sigma}(y),
\label{eqn:}
\end{equation*}
where $C(\sigma)=\{j;a^j(\sigma)\in\N^n\}\subset
\{1,\ldots,n\}$ and $\varphi_{\sigma}$ is a 
$C^{\infty}$ function defined on a neighborhood of the origin. 
Notice that if $\varphi$ is convenient, then 
$\tilde{l}(a^j(\sigma))=0$ for $a\in \Z_+^n\setminus \N^n$. 
Thus, we obtain the expression (\ref{eqn:5.18}) for 
every $C^{\infty}$ function $\varphi$.

The case of real analytic $\varphi$ is easier, so
the proof is omitted. 
\end{proof}

Using the map $\pi: Y_{\Sigma}\to\R^n$ with $x=\pi(y)$ and 
the cut-off functions $\{\chi_{\sigma};\sigma\in\Sigma^{(n)}\}$ 
in the proof of Theorem~5.1 and substituting (\ref{eqn:5.18}), 
we have
\begin{eqnarray*}
&&
Z_{\pm}(s)=\int_{\R^n}f(x)_{\pm}^s \varphi(x)\chi(x)dx \\
&&
\quad=\int_{Y_{\Sigma}} ((f\circ\pi)(y))_{\pm}^s (\varphi\circ\pi)(y) 
(\chi\circ\pi)(y)
|J_{\pi}(y)|dy \\
&&
\quad=\sum_{\sigma\in\Sigma^{(n)}} Z_{\pm}^{(\sigma)}(s),
\end{eqnarray*}
with
\begin{equation}
\begin{split}
&Z_{\pm}^{(\sigma)}(s)
=\int_{\R^n} ((f\circ\pi(\sigma))(y))_{\pm}^s 
(\varphi\circ\pi(\sigma))(y)
(\chi\circ\pi(\sigma))(y)\chi_{\sigma}(y)
|J_{\pi(\sigma)}(y)|dy \\
&\quad 
=\int_{\R^n} 
\left(
\prod_{j=1}^n y_j^{l(a^j(\sigma))}f_{\sigma}(y)
\right)_{\pm}^s 
\left(
\prod_{j=1}^n y_j^{\tilde{l}(a^j(\sigma))}
\varphi_{\sigma}(y)
\right) 
\left|
\prod_{j=1}^n y_j^{ \langle a^j(\sigma) \rangle -1}
\right|\tilde{\chi}_{\sigma}(y)dy, 
\end{split}
\label{eqn:5.19}
\end{equation}
where $\tilde{\chi}_{\sigma}(y)
=(\chi\circ\pi(\sigma))(y)\chi_{\sigma}(y)$.
By an argument similar to that in the proof of Theorem~5.1, 
we see that the poles of $Z_{\pm}^{(\sigma)}(s)$ are contained in 
the set 
\begin{equation}
\left\{
-\frac{\tilde{l}(a^j(\sigma))+ \langle a^j(\sigma) \rangle +\nu}
{l(a^j(\sigma))}
; j\in B(\sigma), \,\, \nu\in\Z_+
\right\}\cup (-\N),
\label{eqn:5.20}
\end{equation}
where 
$B(\sigma)$ is as in (\ref{eqn:5.7}).
Note that the set $(-\N)$ is necessary to express the 
poles because of the same reason as in Remark 5.3.

Next, consider a geometrical meaning of the largest element
of the union of the first set in (\ref{eqn:5.20}) for all 
$\sigma\in\Sigma^{(n)}$. 
If $\varphi$ is convenient or 
real analytic near the origin, then
$\tilde{l}(a)=\min\{ \langle a,\a \rangle ;\a\in\Gamma_+(\varphi)\}$
if $a\in \Z_+^n$.
Therefore, from the definitions of 
$\beta(\cdot)$ in (\ref{eqn:5.10})
and Proposition 5.4, 
we have 
\begin{equation}
\begin{split}
&\quad\max\left\{-\frac{\tilde{l}(a)+ \langle a \rangle }{l(a)};
a\in\tilde{\Sigma}^{(1)}\right\}\\
&=\max\left\{-\frac{ \langle a,\a+\1 \rangle }{l(a)};
\a\in\Gamma_+(\varphi), a\in\tilde{\Sigma}^{(1)}\right\}\\
&=\max\left\{\beta(\a);\a\in\Gamma_+(\varphi)\right\}\\
&=\max\left\{-\frac{1}{d(f,x^{\a})};
\a\in\Gamma_+(\varphi)\right\}\\
&=\max\left\{-\frac{1}{d(f,x^{\a})};
\a\in\Gamma(\varphi,f)\right\}
=-\frac{1}{d(f,\varphi)}.
\end{split}
\label{eqn:5.21}
\end{equation}
On the other hand, when $f$ is convenient, it is easy 
to see $\tilde{\Sigma}^{(1)}=\Sigma^{(1)}\cap \N^n$, 
which implies the first equality in (\ref{eqn:5.21}).
Note that $\tilde{\Sigma}^{(1)}\supset\Sigma^{(1)}\cap \N^n$
in general case.
From (\ref{eqn:5.20}) and (\ref{eqn:5.21}), we see that 
the poles of $Z_{\pm}(s)$
are contained in the set in the theorem. 

Finally, consider the orders of 
the poles of $Z_{\pm}(s)$
at $s=-1/d(f,\varphi)$. 
Considering the construction of simplicial subdivision
in the beginning of the proof, we see that 
for each $\sigma\in\Sigma^{(n)}$ there exists 
a vertex $\alpha_{\sigma}$ of $\Gamma_+(\varphi)$ 
such that 
$ \langle a^j(\sigma),\alpha_{\sigma} \rangle =\tilde{l}(a^j(\sigma))$
for $j=1,\ldots,n$. 
Therefore, for $\sigma\in\Sigma^{(n)}$, we have
\begin{equation}
\begin{split}
&\quad \card\left\{j;
-\frac{\tilde{l}(a^j(\sigma))+ \langle a^j(\sigma) \rangle }{l(a^j(\sigma))}
=-\frac{1}{d(f,\varphi)}
\right\}\\
&=\card\left\{j;d(f,\varphi)
( \langle a^j(\sigma),\a_{\sigma} \rangle + \langle a^j(\sigma) \rangle )
=l(a^j(\sigma))\right\}\\
&=\card\left\{j; \langle a^j(\sigma),d(f,\varphi)(\a_{\sigma}+\1) \rangle 
=l(a^j(\sigma))\right\}\\
&=\card\left\{j;d(f,\varphi)(\a_{\sigma}+\1)\in
H(a^j(\sigma),l(a^j(\sigma)))\right\}\\
&=\rho_f(d(f,\varphi)(\a_{\sigma}+\1)).
\end{split}
\end{equation}
The last equality follows from the last part of the proof 
of Proposition~5.4. 
Since $\rho_f(d(f,\varphi)(\a_{\sigma}+\1))\leq m(f,\varphi)$, 
we obtain the estimate of the orders of the poles 
in Theorem~5.10
by the same argument as that in Proposition~5.4.

It is easy to show the case of $Z(s)$ 
by using the relationship:
$Z(s)=Z_+(s)+Z_-(s)$.
\end{proof}
\begin{remark}
In the beginning of the above proof, we 
constructed a simplicial subdivision $\Sigma$ 
from a more complicated process. 
Before estimating the orders of poles, 
the usual subdivision as in Theorem~5.1 is sufficient.
\end{remark}

Next, when $d(f,\varphi)>1$, 
we consider the coefficients of  
$(s+1/d(f,\varphi))^{-m(f,\varphi)}$ 
in the Laurent expansions of 
$Z_{\pm}(s)$ and $Z(s)$. 
Let 
\begin{equation}
\begin{split} 
&
C_{\pm}=\lim_{s\to-1/d(f,\varphi)} 
(s+1/d(f,\varphi))^{m(f,\varphi)} 
Z_{\pm}(s), \\ 
&
C=\lim_{s\to-1/d(f,\varphi)} 
(s+1/d(f,\varphi))^{m(f,\varphi)} 
Z(s),
\end{split}
\label{eqn:5.22}
\end{equation}
respectively.

\begin{theorem}
Suppose that  
{\rm (i)} 
$f$ is convenient and nondegenerate over $\R$ 
with respect to its Newton polyhedron,
{\rm (ii)}  
$\varphi_{\Gamma_0}$ is nonnegative 
$($resp. nonpositive$)$
on a neighborhood of the origin and 
{\rm (iii)} 
$d(f,\varphi)>1$. 
If the support of $\chi$ is
contained in a sufficiently small neighborhood 
of the origin, 
then $C_+$ and $C_-$ 
are nonnegative $($resp. nonpositive$)$ and 
$C=C_{+}+C_{-}$ is positive $($resp. negative$)$.  
\end{theorem}
\begin{proof}
We only show the theorem in the nonnegative case in 
the assumption (ii). 
 
Let $\Sigma_0$ be the fan constructed from the Newton polyhedron 
of $f$. 
Fix a simplicial subdivision $\Sigma$ of $\Sigma_0$ and 
let $(Y_{\Sigma},\pi)$ be the real resolution 
associated with $\Sigma$ as in Section 4. 

Notice that $\Gamma(\varphi,f)$ and $\Gamma_0$ are compact sets, 
because $f$ is convenient. 
Recall that the essential set 
$\Gamma_0$ is expressed as the disjoint 
union of some finite faces 
$\gamma_1,\ldots,\gamma_l$ of $\Gamma_+(\varphi)$
and that 
the nonnegativity of $\varphi$ is equivalent to 
that of $\varphi_{\gamma_{\mu}}$ 
for $\mu=1,\ldots,l$ (see Section 3). 


We define the functions 
$\varphi_1,\ldots,\varphi_{l+2}$
as follows. 
\begin{equation}
\begin{split}
&\varphi_{\mu}(x)=\varphi_{\gamma_{\mu}} (x) \quad 
\mbox{for $\mu=1,\ldots,l$}, \\
&\varphi_{l+1}(x)=
\varphi_{\Gamma(\varphi,f)\setminus\Gamma_0}(x),\quad\quad
\varphi_{l+2}(x)=
\varphi(x)-\sum_{\mu=1}^{l+1} \varphi_{\mu}(x).
\end{split}
\label{eqn:5.23}
\end{equation}
Substituting (\ref{eqn:5.23}) into 
(\ref{eqn:5.1}), we obtain
$$
Z_{\pm}(s)=\sum_{\mu=1}^{l+2} Z_{\mu,\pm}(s),
$$
where 
$$
Z_{\mu,\pm}(s)=
\int_{\R^n} f(x)_{\pm}^s \varphi_{\mu}(x)\chi(x)dx, 
\quad\,\, \mu=1,\ldots,l+2.
$$

From Theorem~5.10, we see that 
the poles of $Z_{l+2,\pm}(s)$ are contained in the set 
$$
\left\{s\in\C;
{\rm Re}(s)\leq -\frac{1}{d(f,\varphi_{l+2})}
\right\}.
$$
Since $d(f,\varphi_{l+2})<d(f,\varphi)$, 
$Z_{l+2,\pm}(s)$ can be extended analytically to 
the region ${\rm Re}(s)\geq -1/d(f,\varphi)-\delta$
with small $\delta>0$.
Moreover, 
the order of the poles of $Z_{l+1,\pm}(s)$ 
at $s=-1/d(f,\varphi)$ 
is less than $m(f,\varphi)$ from Theorem~5.7
(or Theorem~5.10). 
It suffices to consider 
the poles of $Z_{\mu,\pm}(s)$ at $s=-1/d(f,\varphi)$ 
for $\mu=1,\ldots,l$. 

For each $\sigma\in\Sigma^{(n)}$ and 
$\mu\in\{1,\ldots,l\}$, 
there exists a set 
$A_{\mu}(\sigma)\subset\{1,\ldots,n\}$
such that 
\begin{equation}
A_{\mu}(\sigma)=\left\{
j\in B(\sigma); 
\frac{ \langle a^j(\sigma),\a+\1 \rangle }{l(a^j(\sigma))}=
\frac{1}{d(f,\varphi)} 
\right\} \subset \{1,\ldots,n\},
\end{equation}
for all $\a\in\gamma_{\mu}\cap\Z_+^n$,
where $B(\sigma)$ is as in (\ref{eqn:5.7}).   
For $\mu$, define
$$
\Sigma_{\mu}^{(n)}=\{\sigma\in\Sigma^{(n)};
\card A_{\mu}(\sigma)=m(f,\varphi)\}.
$$
By applying the argument in the proof of 
Theorem~5.1 with the condition $d(f,\varphi)>1$, 
it suffices to consider the pole at 
$s=-1/d(f,\varphi)$ of the function
\begin{equation}
I_{\mu,\sigma,\epsilon}(s)= 
\int_{\R^n}
\Phi_{\mu,\epsilon}(y,s) |f_{\sigma}(y)|^s \psi(y)dy, 
\label{eqn:5.24}
\end{equation}
with
\begin{equation}
\Phi_{\mu,\epsilon}(y,s)
=\left(\prod_{j=1}^n 
(y_j)_{\epsilon_j}^{l(a^j(\sigma))s}
\right)
\left(
\sum_{\a\in\gamma_{\mu}} c_{\a} 
\prod_{j=1}^n y_j^{\langle a^j(\sigma),\a \rangle}
\right)
\left|
\prod_{j=1}^n y_j^{ \langle a^j(\sigma) \rangle -1}
\right|, 
\label{eqn:5.25}
\end{equation}
where
$\sigma\in\Sigma_{\mu}^{(n)}$, $\epsilon\in\{+,-\}^n$, 
and $\psi(y)$ is a cut-off function on $\R(\sigma)^n$ 
satisfying that ${\rm Supp}(\psi)
\subset{\rm Supp}(\tilde{\chi}_{\sigma})$, 
where $\tilde{\chi}_{\sigma}$ is as in (\ref{eqn:5.3}), 
and its support does not intersect the set 
$\{y\in {\rm Supp}(\tilde{\chi}_{\sigma});
f_{\sigma}(y)=0\}$. 
Indeed, if $\sigma\in\Sigma^{(n)}\setminus
\Sigma_{\mu}^{(n)}$, then the order of the pole of 
$I_{\mu,\sigma,\epsilon}(s)$ is less than $m(f,\varphi)$.

A simple computation gives 
\begin{equation}
\begin{split}
&
\Phi_{\mu,\epsilon}(y,s)=
\sum_{\a\in\gamma_{\mu}} 
(-1)^{g_{\sigma,\alpha}(\epsilon)}
c_{\a}
\left(\prod_{j=1}^n 
(y_j)_{\epsilon_j}^
{l(a^j(\sigma))s+\langle a^j(\sigma),\a+\1\rangle-1}
\right)\\
&\quad\quad\quad
=\left(
\prod_{j\in A_{\mu}(\sigma)} 
(y_j)_{\epsilon_j}^{l(a^j(\sigma)) (s+1/d(f,\varphi))-1}
\right)
\Phi_{\mu,\epsilon}(\tilde{y},s), 
\end{split}
\label{eqn:5.26}
\end{equation} 
where $g_{\sigma,\alpha}(\epsilon)$ is as in (\ref{eqn:gsp})
and 
$\tilde{y}$ is defined by 
$\tilde{y}_j=\epsilon_j 1\in\{\pm 1\}$ 
for $j\in A_{\mu}(\sigma)$; 
$\tilde{y}_j=y_j$ for $j\not\in A_{\mu}(\sigma)$. 
Substituting (\ref{eqn:5.26}) into (\ref{eqn:5.24}), we have 
$$
I_{\mu,\sigma,\epsilon}(s)=
\int_{\R^n}
\left(
\prod_{j\in A_{\mu}(\sigma)} 
(y_j)_{\epsilon_j}^{l(a^j(\sigma)) (s+1/d(f,\varphi))-1}
\right)
\Phi_{\mu,\epsilon}(\tilde{y},s) |f_{\sigma}(y)|^s \psi(y)dy.
$$

First, we consider the case when $m(f,\varphi)<n$.
By applying Lemma~5.6, 
the coefficient of $(s+1/d(f,\varphi))^{-m(f,\varphi)}$  
in the Laurent expansion of $I_{\mu,\sigma,\epsilon}(s)$ is 
\begin{equation}
\frac{1}{\prod_{j\in A_{\mu}(\sigma)}l(a^j(\sigma))} 
\int_{\R^{n-m(f,\varphi)}}
\Phi_{\mu,\epsilon}(\tilde{y},-1/d(f,\varphi) )
|f_{\sigma}(\hat{y})|^{-1/d(f,\varphi)} 
\psi(\hat{y})d\hat{y},
\label{eqn:5.27}
\end{equation}
where $\hat{y}$ is defined by 
$\hat{y}_j=0$ for $j\in A_{\mu}(\sigma)$,  
$\hat{y}_j=y_j$ for $j\not\in A_{\mu}(\sigma)$ and 
$d\hat{y}=\prod_{j\not\in A_{\mu}(\sigma)} dy_j$. 
From (\ref{eqn:5.25}), 
the nonnegativity of $\varphi_{\gamma_{\mu}}$
implies that of $\Phi_{\mu,\epsilon}$. 
Moreover, since $\varphi_{\gamma_{\mu}}$ is a polynomial, 
there is an open set in ${\rm Supp}(\psi)$ such that 
$\Phi_{\mu,\epsilon}(\cdot, -1/d(f,\varphi))$  is   
positive there. 
Consider the case that 
the support of $\psi$ contains the origin. 
Then we can see that (\ref{eqn:5.27}) is positive for any 
$\sigma\in\Sigma_{\mu}^{(n)}$ and $\epsilon\in\{+,-\}^n$.  

Next, we consider the case when $m(f,\varphi)=n$. 
Then, 
$A_{\mu}(\sigma)=B(\sigma)=\{1,\ldots,n\}$. 
Since the essential set $\Gamma_0$ is a set
of vertices of $\Gamma_+(\varphi)$, 
the nonnegative assumption of $\varphi_{\Gamma_0}$ 
implies that 
each $\varphi_{\mu}$ can be expressed as a monomial
of the form: $\varphi_{\mu}(x)=c_{\alpha}x^{\alpha}$, 
where
$c_{\alpha}$ is a positive constant and 
every component of $\alpha$ is even. 
From an easy computation, 
the corresponding coefficient in this case 
is obtained as  
\begin{equation}
\frac{c_{\alpha}|f_{\sigma}(0)|^{-1/d(f,\varphi)}}
{\prod_{j=1}^n l(a^j(\sigma))}. 
\label{eqn:5.271}
\end{equation}
Of course, (\ref{eqn:5.271}) is positive 
for any $\sigma\in\Sigma_{\mu}^{(n)}$ and 
is independent of 
$\epsilon\in\{-,+\}^n$. 

Finally, 
by the same argument as that in the proof of Theorem~5.7, 
we can see the nonnegativity of $C_+$ and $C_-$  
and the positivity of $C=C_++C_-$ in the theorem. 
\end{proof}

\begin{proposition}
Suppose that 
the conditions {\rm (i)}, {\rm (ii)} in 
Theorem~$5.10$ are satisfied 
and 
{\rm (iii)} $d(f,\varphi)<1$. 
Let $1,\ldots,k_*$ be all the natural numbers strictly smaller than 
$1/d(f,\varphi)$. 
If the support of $\chi$ is
contained in a sufficiently small neighborhood 
of the origin, then
$Z_{+}(s)$ and $Z_-(s)$ have 
at $s=-1,\ldots,-k_*$ poles of order not 
higher than 1 and they do not have other poles 
in the region ${\rm Re}(s)>-1/d(f,\varphi)$.    
Moreover, let $a_k^{+}$, $a_k^-$ 
be the residues of 
$Z_{+}(s)$, $Z_-(s)$ at $s=-k$, respectively,  
then we have $a_k^+=(-1)^{k-1}a_k^-$ for $k=1,\ldots,k_*$.
\end{proposition}
\begin{proof}
The difference between (\ref{eqn:5.3}) and (\ref{eqn:5.19}) 
does not essentially affect 
the argument in the proof of Proposition~5.8 . 
The details are left to the readers. 
\end{proof}
\subsection{Remarks}
In this subsection, let us consider Theorem~5.7 
(with Remark~5.9) and Theorem~5.13
under the additional assumption: 
$f$ is nonnegative or nonpositive near the origin. 
The following theorem shows that 
the same assertions can be obtained 
without the assumption: $d(f,\varphi)>1$.  
\begin{theorem}
Suppose that 
{\rm (i)} 
$f$ is nondegenerate over $\R$ with respect to its
Newton polyhedron, 
{\rm (ii)}
$f$ is nonnegative or nonpositive 
on a neighborhood of the origin and 
{\rm (iii)}
at least one of the following condition is satisfied:
\begin{enumerate}
\item[(a)] 
$\varphi$ is expressed as $\varphi(x)=x^p\tilde{\varphi}(x)$
on a neighborhood of the origin, 
where every component of $p\in \Z_+^n$ is even
and 
$\tilde{\varphi}(0)> 0$
$($resp. $\tilde{\varphi}(0)< 0)$;
\item[(b)] 
$f$ is convenient and 
$\varphi_{\Gamma_0}$ is nonnegative 
$($resp. nonpositive$)$
on a neighborhood of the origin.
\end{enumerate}
If the support of $\chi$ is
contained in a sufficiently small neighborhood 
of the origin,  
then $C_+$ and $C_-$ 
are nonnegative $($resp. nonpositive$)$ and 
$C=C_{+}+C_{-}$ is positive $($resp. negative$)$, 
where $C_{\pm}$, $C$ are as in $(\ref{eqn:5.22})$.  
\end{theorem}
\begin{proof}
We only show the case that 
$f$ is nonnegative in the assumption (ii). 

Let $\Sigma_0$ be the fan constructed from the Newton polyhedron 
of $f$. 
Fix a simplicial subdivision $\Sigma$ of $\Sigma_0$ and 
let $(Y_{\Sigma},\pi)$ be the real resolution 
associated with $\Sigma$ as in Section 4. 
Let $\sigma\in \Sigma^{(n)}$. 
First, 
From Proposition~4.2 and the assumptions (i), (ii), 
there exists a neighborhood $U_{\sigma}$ of the origin
such that $f(x)\geq 0$ for $x\in U_{\sigma}$ and 
$$
f(\pi(\sigma)(y))=
\left(\prod_{j=1}^n y_j^{l(a^j(\sigma))}\right)
f_{\sigma}(y) \quad 
\mbox{for $y\in\pi(\sigma)^{-1}(U_{\sigma})$,}
$$
where $f_{\sigma}$ is real analytic and 
$f_{\sigma}(0)>0$.
It is easy to see that 
all $l(a^j(\sigma))$ are even and that 
$f_{\sigma}(y)\geq 0$
for $y\in\pi(\sigma)^{-1}(U_{\sigma})$. 

Now, let us assume that there exists a point 
$y_0\in T_I\cap \pi(\sigma)^{-1}(U_{\sigma})$ 
with nonempty $I\subset \{1,\ldots,n\}$ 
such that $f_{\sigma}(y_0)=0$. 
(See (\ref{eqn:4.3})  
for the definition of $T_I$.) 
Since $f$ is nondegenerate, 
Proposition~4.2 implies that there is a point 
$y_* \in \pi(\sigma)^{-1}(U_{\sigma})$ 
near $y_0$ 
such that $f_{\sigma}(y_*)<0$.
This contradicts the nonnegativity 
of $f$ on $U_{\sigma}$, 
so we see that  
$\{
y\in\pi(\sigma)^{-1}(U_{\sigma});f_{\sigma}(y)=0
\}\subset (\R\setminus\{0\})^n$.

Therefore,
there exists a small neighborhood 
$V_{\sigma}\subset U_{\sigma}$ of the origin such that
$\{y;f_{\sigma}(y)=0\}\cap \pi(\sigma)^{-1}(V_{\sigma})
=\emptyset$.
Then  $f_{\sigma}$ is positive on 
$\pi(\sigma)^{-1}(V_{\sigma})$ for each 
$\sigma\in\Sigma^{(n)}$. 

Let us investigate the properties of 
poles of $Z_{\pm}(s)$ in (\ref{eqn:5.1}), 
where the cut-off function $\chi$ has a support 
contained in the set 
$V:=\bigcap_{\sigma\in\Sigma^{(n)}}V_{\sigma}$.
We apply the arguments 
in the proofs of Theorems~5.7 and 5.13
to these cases. 
In the process of analysis, 
the decomposition similar to (\ref{eqn:5.4})
is obtained, but the functions corresponding to 
$J_{\sigma,\pm}^{(k)}(s)$ do not appear 
because $f_{\sigma}$ is always positive or negative on $V$. 
Thus, since it suffices to consider the poles of 
$I_{\sigma,\pm}^{(k)}(s)$, 
we can easily obtain the assertions of this theorem.  
\end{proof}

\subsection{Certain symmetrical properties}

We denote by 
$\beta_{\pm}(f,\varphi)$, $\hat{\beta}(f,\varphi)$ 
the largest poles 
of $Z_{\pm}(s)$, $Z(s)$ and 
by $\eta_{\pm}(f,\varphi)$, $\hat{\eta}(f,\varphi)$ 
their orders, 
respectively.
\begin{theorem}
Let $f,\varphi$ be nonnegative or nonpositive real analytic
functions defined on a neighborhood of the origin. 
Suppose that $f$ and $\varphi$ are 
convenient and nondegenerate over $\R$ with respect to 
their Newton polyhedra. 
If the support of $\chi$ is
contained in a sufficiently small neighborhood 
of the origin, then we have 
\begin{equation}
\b_{\pm}(x^{\1}f,\varphi)\b_{\pm}(x^{\1}\varphi,f)\leq 1 
\mbox{ and } 
\hat{\b}(x^{\1}f,\varphi)\hat{\b}(x^{\1}\varphi,f)\leq 1
\label{eqn:5.28}
\end{equation}
Moreover, the following two conditions are equivalent: 
\begin{enumerate}
\item The equality holds in each estimate 
in $(\ref{eqn:5.28})$;
\item There exists a positive rational number $d$ such that 
$\Gamma_+(x^{\1} f)=d\cdot\Gamma_+(x^{\1} \varphi)$.
\end{enumerate}
If the condition {\rm (i)} or {\rm (ii)} is satisfied, 
then we have 
$\eta_{\pm}(x^{\1}f,\varphi)=\eta_{\pm}(x^{\1}\varphi,f)
=\hat{\eta}(x^{\1}\varphi,f)=n$.
\end{theorem}
\begin{proof}
Admitting Lemmas~5.17 and~5.19 below, 
we prove the theorem as follows. 
From the assumptions in the theorem and 
Lemma~5.17 with $p=\1$,  
Theorem~5.13 implies the equations 
$\beta(x^{\1}f,\varphi)=-1/d(x^{\1}f,\varphi)$ 
and $\beta(x^{\1}\varphi,f)=-1/d(x^{\1}\varphi,f)$. 
By Lemma~5.19, these equations imply the inequality 
(\ref{eqn:5.28}) and the equivalence of (i) and (ii).
The condition (ii) in the theorem implies 
$\Gamma(\varphi, x^{\1}f)=\Gamma(\varphi)$. 
Thus, $\Gamma_0$ is the set of the vertices  
(i.e., zero-dimensional faces) of 
$\Gamma(\varphi)$. 
This means $\eta_{\pm}(x^{\1}f,\varphi)
=m(x^{\1}f,\varphi)=n$. 
Similarly, we see $\eta_{\pm}(x^{\1}\varphi,f)=n$.
\end{proof}

\begin{lemma}
Let $p\in \Z_+^n$ and $g$ be a $C^{\infty}$ function
defined on a neighborhood of the origin in $\R^n$ such that
$\Gamma(g)\neq\emptyset$. Then, 
$g$ is nondegenerate over $\R$
with respect to its Newton polyhedron 
if and only if so is $x^p g$. 
\end{lemma}
\begin{proof}
First, notice that if a $C^{\infty}$ function $\psi$ 
has a quasihomogeneous property: 
$
\psi(t^{m_1}x_1,\ldots,t^{m_n}x_n) =t^c \psi(x)
$ 
for $t>0$ and $x\in\R^n$, then 
$\nabla \psi (x)=0$ implies $\psi(x)=0$. 
In fact, this follows from Euler's identity:
\begin{equation*}
m_1x_1\frac{\partial \psi}{\partial x_1}(x)
+\cdots+
m_nx_n\frac{\partial \psi}{\partial x_n}(x)
=c\psi(x)
\end{equation*}
for $x\in\R^n$.

Now, 
we assume that $h(x):=x^p g(x)$ is not nondegenerate
in the above sense.
Then, there exist a face $\gamma$ of 
$\Gamma_+(h)$ and a point $x_0 \in(\R\setminus\{0\})^n$ 
such that $\nabla h_{\gamma}(x_0)=0$ and 
$h_{\gamma}(x_0)=0$. 
It is easy to see the existence of the face
$\tilde{\gamma}$ of $\Gamma_+(g)$ such that 
$h_{\gamma}(x)=x^p g_{\tilde{\gamma}}(x)$. 
Note that $\tilde{\gamma}+p=\gamma$.
The assumption on $h$ implies 
$\nabla g_{\tilde{\gamma}}(x_0)=0$,  
which means that $g$ is not nondegenerate.

The above argument is also 
available for $p\in (-\Z_+)^n$, 
so the converse can be shown similarly. 
\end{proof}
\begin{remark}
By observing the above proof, it might be
expected that the above lemma can be generalized 
as follow. 
``Let $g$ be the same as above and 
let a $C^{\infty}$ function $\rho$ be positive on 
$(\R\setminus\{0\})^n$. 
Then $g$ is nondegenerate in the above sense 
if and only if so is $\rho(x)g(x)$.'' 
But, this claim is not true.  
In fact, consider the two-dimensional case: 
$g(x,y)=xy$ and $\rho(x,y)=(x-y)^2+x^4$.
\end{remark}
\begin{lemma}
Let $g$,$h$ be $C^{\infty}$ functions defined 
near the origin, then 
we have 
$$d(x^{\1}g,h) d(x^{\1}h,g)\geq 1.$$
The equality holds in the above, if and only if 
there exists a positive rational number $d$ such that
$\Gamma_+(x^{\1} g)=d\cdot\Gamma_+(x^{\1} h)$.
\end{lemma}
\begin{proof}
From the definition of $d(\cdot,\cdot)$, 
\begin{eqnarray}
&&
d(x^{\1}g,h)
\cdot\Gamma_+(x^{\1}h)\subset
\Gamma_+(x^{\1} g),
\label{eqn:5.31}
\\
&&
d(x^{\1}h,g)
\cdot\Gamma_+(x^{\1}g)\subset
\Gamma_+(x^{\1} h).
\label{eqn:5.32}
\end{eqnarray}
Putting (\ref{eqn:5.31}),(\ref{eqn:5.32}) together, 
we have 
\begin{equation}
(d(x^{\1}g,h) d(x^{\1}h,g))
\cdot\Gamma_+(x^{\1}g)\subset
\Gamma_+(x^{\1} g).
\label{eqn:5.33}
\end{equation}
This implies 
$d(x^{\1}g,h) d(x^{\1}h,g)\geq 1$.

If there exists a positive number $d$ such that 
$\Gamma_+(x^{\1} g)=d\cdot\Gamma_+(x^{\1}h)$, 
then 
it is clear that 
$d=d(x^{\1}g,h)=1/d(x^{\1}h,g)$. 
On the other hand, if 
$\Gamma_+(x^{\1} g)\neq d\cdot\Gamma_+(x^{\1} h)$
for any $d>0$, then the inclusions 
in (\ref{eqn:5.31}), (\ref{eqn:5.32})
are in the strict sense, therefore so is the inclusion 
in (\ref{eqn:5.33}). 
Then we see that 
$d(x^{\1}g,h) d(x^{\1}h,g)>1.$
\end{proof}

\section{Proofs of the theorems in Section 2}

\subsection{Relationship between $I(\tau)$ and 
$Z_{\pm}(s)$}

It is known 
(see \cite{igu78}, \cite{kan81}, \cite{agv88}, etc.) 
that 
the study of the asymptotic behavior of the 
oscillatory integral $I(\tau)$ in (\ref{eqn:1.1})
can be reduced to an investigation of the poles
of the functions $Z_{\pm}(s)$ in (\ref{eqn:5.1}). 
Here, we explain an outline of 
these situations. 
Let $f$,$\varphi$,$\chi$ satisfy 
the conditions (A),(B),(C) in Section~2.2. 
Suppose that the support of $\chi$ is sufficiently small. 

Define the Gelfand-Leray function: 
$K:\R \to \R$ as
\begin{equation}
K(t)=\int_{W_t} \varphi(x) \chi(x) \omega, 
\label{eqn:6.3}
\end{equation}
where $W_t=\{x\in\R^n; f(x)=t\}$ and $\omega$ is 
the surface element on $W_t$ which is determined by 
$
df\wedge \omega=dx_1\wedge\cdots\wedge dx_n.
$
$I(\tau)$ and $Z_{\pm}(s)$ can be expressed by 
using $K(t)$: 
Changing the integral variables 
in (\ref{eqn:1.1}),(\ref{eqn:5.1}), we have 
\begin{equation}
\begin{split}
&I(\tau)=\int_{-\infty}^{\infty}
e^{i\tau t}K(t)dt 
=\int_{0}^{\infty}
e^{i\tau t}K(t)dt+
\int_{0}^{\infty}
e^{-i\tau t}K(-t)dt,
\end{split}
\label{eqn:6.2}
\end{equation}
\begin{equation}
Z_{\pm}(s)
=\int_0^{\infty} t^s K(\pm t)dt,
\label{eqn:6.3}
\end{equation}
respectively. 
Applying the inverse formula of the Mellin transform
to (\ref{eqn:6.3}),  
we have 
\begin{equation}
K(\pm t)
=\frac{1}{2\pi i}
\int_{c-i\infty}^{c+i\infty} 
Z_{\pm}(s)t^{-s-1}ds,
\label{eqn:6.4}
\end{equation}
where $c>0$ and the integral contour follows
the line Re$(s)=c$ upwards. 
Recall that 
$Z_{+}(s)$ and $Z_{-}(s)$ are meromorphic functions and 
their poles exist on the negative part of the real axis. 
By deforming the integral contour as $c$ tends 
to $-\infty$ in (\ref{eqn:6.4}), 
the residue formula gives the
asymptotic expansions of $K(t)$ 
as $t\to\pm 0$.
Substituting these expansions of $K(t)$ into 
(\ref{eqn:6.2}), 
we can get an asymptotic expansion of $I(\tau)$ 
as $\tau\to+\infty$.

Through the above calculation, 
we see more precise relationship for the coefficients. 
If $Z_{+}(s)$ and $Z_{-}(s)$ 
have the Laurent expansions at $s=-\lambda$:
\begin{equation*}
Z_{\pm}(s)=\frac{B_{\pm}}{(s+\lambda)^{\rho}}+
O\left(\frac{1}{(s+\lambda)^{\rho-1}}\right),
\end{equation*}
respectively, 
then the corresponding part in the asymptotic
expansion of $I(\tau)$ has the form
$$
B\tau^{-\lambda}(\log \tau)^{\rho-1}+
O(\tau^{-\lambda}(\log \tau)^{\rho-2}).
$$ 
Here a simple computation gives 
the following relationship:
\begin{equation}
B=\frac{\Gamma(\lambda)}{(\rho-1)!}
\left[
e^{i\pi \lambda/2}B_+ +e^{-i\pi \lambda/2}B_-
\right],
\label{eqn:6.5}
\end{equation}
where $\Gamma$ is the Gamma function. 
\subsection{Proofs of Theorems 2.2, 2.7 and 2.12}
Applying the above argument to the results 
relating to $Z_{\pm}(s)$ in Section~5, 
we obtain the theorems in Section~2. 

{\it Proof of Theorem~2.2.}\quad
This theorem follows from Theorem~5.1 
with Remark~5.9 and Theorem~5.10. 
Notice that Propositions~5.8 and 5.14 and 
the relationship (\ref{eqn:6.5}) 
induce the cancellation of the coefficients 
of the term, 
whose orders are larger than $-1/d(f,\varphi)$. 
The estimate in Remark~2.3 is 
also obtained by using the estimates of 
orders of the poles of $Z_{\pm}(s),Z(s)$ 
in Proposition~5.4 and Theorem~5.10. 

{\it Proof of Theorem~2.7.}\quad 
This theorem follows from Theorem~5.7
with Remark~5.9 and Theorems~5.13 and 5.15. 
Notice that the relationship (\ref{eqn:6.5}) 
gives the information about the
coefficient of the leading term of $I(\tau)$. 

{\it Proof of Theorem~2.12.}\quad 
This theorem follows from Theorem~5.16. 

\section{Examples}
In this section, we give some examples of 
the phase and the amplitude in the integral 
(\ref{eqn:1.1}), 
which clarifies the subtlety of 
our results in Sections~2 and 6.
Throughout this section, we always assume that 
$f$, $\varphi$, $\chi$ 
satisfy the conditions (A), (B), (C)
in Section~2. 
(In Examples~1,~2,  each $f$, $\varphi$, $\varphi_t$ 
satisfies the respective condition.)

\subsection{The one-dimensional case}
Let us compute the asymptotic expansion 
of $I(\tau)$ in (\ref{eqn:1.1}) as $\tau\to+\infty$ 
in the one-dimensional case by using our analysis in this paper. 
As mentioned in Section~2, 
the results below 
can also be obtained by using the analysis in \cite{ste93}.
Note that the computation below is valid for 
$C^{\infty}$ phases. 
From the assumptions
$\Gamma_+(f),\Gamma_+(\varphi)\neq \emptyset$, 
$f$, $\varphi$ can be expressed as 
\begin{equation*}
f(x)=x^{q}\tilde{f}(x), \quad
\varphi(x)=x^p\tilde{\varphi}(x),
\end{equation*}
where $q,p\in\Z_+$, $q\geq 2$ and 
$\tilde{f},\tilde{\varphi}$ are 
$C^{\infty}$ functions defined on a 
neighborhood of the origin with 
$\tilde{f}(0)\tilde{\varphi}(0)\neq 0$. 
Suppose that the support of $\chi$ is so small
that $\tilde{f},\tilde{\varphi}$ do not have
any zero on the support.

It is easy to see that 
$f$ is nondegenerate over $\R$ with respect to 
its Newton polyhedron, 
$\Gamma_+(f)=[q,\infty)$, 
$\Gamma_+(\varphi)=[p,\infty)$, 
$d(f,\varphi)=\frac{q}{p+1}$ and $m(f,\varphi)=1$. 
Let $\alpha$ be the sign 
of $\tilde{f}(x)$ on the support of $\chi$.
From a simple computation, for even $q$   
\begin{equation}
\begin{split}
&Z_{\alpha}(s)=\int_0^{\infty} x^{qs+p}
\{|\tilde{f}(x)|^s\tilde{\varphi}(x)\chi(x)+
(-1)^p|\tilde{f}(-x)|^s\tilde{\varphi}(-x)\chi(-x)\}dx, 
\\
&
Z_{-\alpha}(s)=0,
\end{split}
\end{equation}
and for odd $q$
\begin{equation}
\begin{split}
&Z_{\alpha}(s)=
\int_0^{\infty} x^{qs+p}
|\tilde{f}(x)|^s\tilde{\varphi}(x)\chi(x)dx,\\
&Z_{-\alpha}(s)=
(-1)^p
\int_0^{\infty}x^{qs+p}
|\tilde{f}(-x)|^s\tilde{\varphi}(-x)\chi(-x)dx.
\end{split}
\end{equation}
By using Lemma~5.2, we can see that 
the poles of $Z_{\pm}(s)$ are simple and they 
are contained in the set 
$\{-\frac{p+1+\nu}{q};\nu\in\Z_+\}$.
By using Lemma~5.6, 
we can compute the explicit values 
of the coefficients of the term $(s+\frac{p+1}{q})^{-1}$ 
in the Laurent expansions of $Z_{+}(s)$ and $Z_-(s)$. 

Next, applying the argument in Section~6.1, we have  
$$
I(\tau)
\sim\tau^{-\frac{p+1}{q}}
\sum_{j=0}^{\infty} C_j \tau^{-j/q}
 \quad \mbox{as $\tau\to\infty$.}
$$ 
The relationship (\ref{eqn:6.5}) gives the values
of the coefficient $C_0$. 
As a result, we can see all the cases that 
$\beta(f,\varphi)=-1/d(f,\varphi)$ holds. 
\begin{enumerate}
\item ($q$:even; $p$:even)  
$C_0=\frac{2}{q}\Gamma\left(\frac{p+1}{q}\right)
|\tilde{f}(0)|^{-\frac{p+1}{q}}\tilde{\varphi}(0)
e^{\alpha i\frac{p+1}{2q}\pi}\neq 0$, 
which implies $\beta(f,\varphi)=-1/d(f,\varphi)$;
\item ($q$:even; $p$:odd) 
$C_0=0$, which implies 
$\beta(f,\varphi)<-1/d(f,\varphi)$;
\item ($q$:odd; $p$:even) 
$C_0=
\frac{2}{q}\Gamma\left(\frac{p+1}{q}\right)
|\tilde{f}(0)|^{-\frac{p+1}{q}}\tilde{\varphi}(0)
\cos\left(\frac{p+1}{2q}\pi\right)$, 
which implies that 
$\beta(f,\varphi)=-1/d(f,\varphi)$ 
is equivalent to 
$\frac{p+1}{2q}\not\in \N +\frac{1}{2}$; 
\item ($q$:odd; $p$:odd) 
$C_0=
\alpha\frac{2i}{q}\Gamma\left(\frac{p+1}{q}\right)
|\tilde{f}(0)|^{-\frac{p+1}{q}}\tilde{\varphi}(0)
\sin\left(\frac{p+1}{2q}\pi\right)$, 
which implies that 
$\beta(f,\varphi)=-1/d(f,\varphi)$ is equivalent to 
$\frac{p+1}{2q}\not\in \N$. 
\end{enumerate}
Let us compare the conditions (a),(b),(c),(d) 
in Theorem~2.7 
with the condition of $p,q$. 
That $q$ (resp. $p$) is even  
is equivalent to the condition (b) (resp. (c), (d)).  
The condition (a) is equivalent to the inequality: 
$\frac{p+1}{2q}\pi<\frac{\pi}{2}$, 
which implies $C_0\neq 0$ in (iii).

\subsection{Example 1}
Consider the following two-dimensional example: 
\begin{equation*}
\begin{split}
&f(x_1,x_2)=x_1^4,
\\
&\varphi(x_1,x_2)=x_1^2x_2^2+e^{-1/x_2^2}
(=:\varphi_1(x_1,x_2)+\varphi_2(x_1,x_2)),
\end{split}
\end{equation*}
and $\chi$ is radially symmetric about the origin. 
It is easy to see that
$f$ is nondegenerate over $\R$ with respect to 
its Newton polyhedron,
$\Gamma_+(f)=\{(4,0)\}+\R_+^2$,
$\Gamma_+(\varphi)=\Gamma_+(\varphi_1)
=\{(2,2)\}+\R_+^2$, $\Gamma_+(\varphi_2)=\emptyset$, 
$d(f,\varphi)=4/3$, $m(f,\varphi)=1$.
Define 
$$
Z_{\pm}^{(j)}(s)=\int_{\R^2} 
(f(x))_{\pm}^s\varphi_j(x)\chi(x)dx \quad\,\,
j=1,2.
$$ 
Note $Z_-(s)=0$. 
A simple computation gives 
$$
Z_+^{(1)}(s)=
4\int_0^{\infty}\int_0^{\infty} 
x_1^{4s+2}x_2^2 \chi(x_1,x_2)dx_1dx_2.
$$
By Lemma~5.2, we see that the poles of $Z_+^{(1)}(s)$
are simple and they 
are contained in the set 
$\{-3/4,-4/4,-5/4,\ldots\}$. 
Similarly, the poles of 
$$
Z_+^{(2)}(s)=
4\int_0^{\infty}\int_0^{\infty} 
x_1^{4s} e^{-1/x_2^2}\chi(x_1,x_2)dx_1dx_2
$$
are simple and contained in the set 
$\{-1/4,-2/4,-3/4,\ldots\}$.
Moreover, Lemma~5.6 implies that 
the coefficient of $(s+1/4)^{-1}$ is 
$$
\int_0^{\infty} e^{-1/x_2^2}\chi(0,x_2)dx_2>0.
$$
Therefore, we have 
$\beta_+(f,\varphi)=\beta(f,\varphi)=-1/4$. 
As a result, 
$\beta(f,\varphi)> -1/d(f,\varphi) (=-3/4)$. 

This example does not satisfy the 
condition (d) in Theorem~2.2. 
Noticing that 
$\Gamma_+(\varphi)=\{(2,2)\}+\R_+^2$, 
we see that the information of the Newton polyhedron 
is not sufficient to understand the behavior of 
oscillatory integrals 
in the case of $C^{\infty}$ amplitudes. 

\subsection{Example 2}
Consider the following two-dimensional example with
a real parameter $t$:
\begin{equation*}
\begin{split}
&f(x_1,x_2)=x_1^5+x_1^6+x_2^5,
\\
&\varphi_t(x_1,x_2)=
x_1^2+tx_1 x_2+x_2^2.
\end{split}
\end{equation*}
It is easy to see that
$f$ is nondegenerate over $\R$ with respect to 
its Newton polyhedron,
$(\varphi_t)_{\Gamma_0}(x)=\varphi_t(x)$, 
$d(f,\varphi_t)=5/4$, 
and $m(f,\varphi_t)=1$. 
$(\varphi_t)_{\Gamma_0}(x)$ 
is nonnegative on $\R^2$, 
if and only if $|t|\leq 2$. 
Thus, 
Theorem~5.7 implies that 
$\beta(f,\varphi_t)=-1/d(f,\varphi_t)=-4/5$ 
if $|t|\leq 2$. 
In this example, we understand the situation 
in more detail from the explicit 
computation below. 

By applying the computation in Section 5, 
we see the properties of poles of the functions 
$Z_+(s)$ and $Z_-(s)$ in the following. 
The poles of the functions $Z_+(s)$ and $Z_-(s)$ 
are contained in the set 
$\{-4/5,-5/5,-6/5,\ldots\}$ and 
their order is at most one. 
Let 
$C_+(t)$, $C_-(t)$ be the coefficients of  
$(s-4/5)^{-1}$ in the Laurent expansions of 
$Z_+(s)$ and $Z_-(s)$. 
Then, we have
$
C_+(t)=C_-(t)=A+tB
$
with 
$$
A:=\frac{1}{5}\int_{-\infty}^{\infty} 
|u^5+1|^{-4/5}(u^2+1)du, \quad 
B:=\frac{1}{5}\int_{-\infty}^{\infty} 
|u^5+1|^{-4/5}udu.
$$
Note that $A$ is positive and $B$ is negative.

Next, applying the argument in Section~6.1, 
$I(\tau)$ has the asymptotic expansion of the form:
\begin{equation}
I(\tau)\sim \tau^{-\frac{4}{5}}
\sum_{j=0}^{\infty}C_j(t) \tau^{-j/5} \quad 
\mbox{as $\tau\to+\infty$.}
\label{eqn:7.3}
\end{equation}
The relationship (\ref{eqn:6.5}) gives
$C_0(t)=
2\Gamma(\frac{4}{5})\cos(\frac{2}{5}\pi)
(A+tB).$

Set $t_0=-A/B(>0)$.    
From the above value of $C_0(t)$, if $t\neq t_0$, then 
the equation $\beta(f,\varphi_t)=-1/d(f,\varphi_t)$ 
holds.
This means that the condition (d) in 
Theorem~2.7 is not necessary 
to satisfy the above equation.  
Furthermore, this example shows that 
the oscillation index is 
determined by not only the geometry of 
the Newton polyhedra 
but also the values of the coefficients of 
$x^{\alpha}$ for $\alpha\in\Gamma_0$ 
in the Taylor expansion of the amplitude. 
\begin{note}
The existence of the term $x_1^6$ in $f$ 
produces 
infinitely many non-zero coefficients 
$C_j(t)$ in the asymptotic expansion (\ref{eqn:7.3})
for any $t$. 
\end{note}
\subsection{Comments on results in \cite{agv88}}
As mentioned in the Introduction, 
there have been studies in \cite{agv88} 
in a similar direction to our investigations.  
In our language, their results can be stated as follows. 

\begin{theorem?}[Theorem 8.4 in \cite{agv88}, p 254]
If 
$f$ is nondegenerate over $\R$ 
with respect to its Newton polyhedron,
then 
\begin{enumerate}
\item $\beta(f,\varphi)\leq -1/d(f,\varphi)$;
\item If $d(f,\varphi)>1$ and 
$\Gamma_+(\varphi)=\{p\}+\R_+^n$ with $p\in\Z_+^n$, 
then $\beta(f,\varphi)= -1/d(f,\varphi).$
\end{enumerate}
\end{theorem?}

Unfortunately, more additional assumptions are 
necessary to obtain the above assertions (i), (ii). 
Indeed, it is easy to see that 
Example 1 violates (i), (ii). 
As for (ii), even if $\varphi$ is real analytic, 
the one-dimensional case in Section~7.1 indicates 
that at least some condition on the power $p$ 
is needed. 
(It is easy 
to find counterexamples in higher dimensional case.) 
The same case shows that 
the evenness of $p$ is not always necessary 
to satisfy $\beta(f,\varphi)= -1/d(f,\varphi)$. 


\vspace{1 em}

{\sc Acknowledgements.}\quad 
The authors would like to express their sincere gratitude 
to the referee for his/her careful reading of the manuscript
and giving the authors many valuable comments. 


\end{document}